\documentclass[12pt]{amsart}
\usepackage{times, pdfpages, graphicx}
\usepackage{comment}
\usepackage{float}

\usepackage{exscale}
\usepackage{epsfig}
\usepackage{cancel}

\usepackage{amsthm}
\usepackage{amsfonts}

\usepackage{graphics,psfrag,graphicx,subfigure,epsfig, xypic}
\usepackage{amsmath,amsfonts,amssymb,amstext,amsthm,amscd,mathrsfs}
\usepackage[english]{babel}
\usepackage[all]{xy}
\usepackage{multicol}
\usepackage{rotating}

\usepackage{amsmath, amsfonts, amssymb, amscd, enumerate}
\usepackage{color}
\usepackage[all]{xy}

\theoremstyle{plain}
 \newtheorem{theo}{Theorem}[section]

\theoremstyle{plain}
\newtheorem{corollary}[theo]{Corollary}
\newtheorem{proposition}[theo]{Proposition}
\newtheorem{definition}[theo]{Definition}
\newtheorem{remark}[theo]{Remark}
\newtheorem{example}[theo]{Example}

\newcommand{\beq}{\begin{equation}}
\newcommand{\eeq}{\end{equation}}

\newcommand{\Arg}{\textnormal{Arg}}

\newcommand{\C}{\mathbb{C}}
\newcommand{\bR}{\mathbb{R}}

\newcommand{\R}{\mathbb{R}}
\renewcommand{\H}{\mathbb{H}}
\newcommand{\Z}{\mathbb{Z}}
\renewcommand{\P}{\mathbb{P}}

\newcommand{\bK}{\mathbb{K}}

\newcommand{\bS}{\mathbb{S}}

\newcommand{\pro}{{\mathrm{pr}_1}}

\newcommand{\cI}{\mathcal{I}}

\newcommand{\ra}{\rightarrow}

\renewcommand{\square}{\kern1pt\vbox
{\hrule height 0.6pt\hbox{\vrule width 0.6pt\hskip 3pt
\vbox{\vskip 6pt}\hskip 3pt\vrule width 0.6pt}\hrule height0.6pt}\kern1pt}

\DeclareMathOperator\Id{Id}

\renewcommand\Re{\operatorname{Re}}
\renewcommand\Im{\operatorname{Im}}

\renewcommand{\Re}{{\rm Re}}
\renewcommand{\Im}{{\rm Im}}

\newcommand{\sign}{\mathrm{sign}}

\newcommand{\be}{\begin{equation}}
\newcommand{\ee}{\end{equation}}

\def\<#1,#2>{\langle\,#1,\,#2\,\rangle}
\newcommand{\arr}{\begin{array}{rlll}}
\newcommand{\ea}{\end{array}}
\newcommand{\bea}{\begin{eqnarray}}
\newcommand{\eea}{\end{eqnarray}}
\newcommand{\bean}{\begin{eqnarray*}}
\newcommand{\eean}{\end{eqnarray*}}

\def\sideremark#1{\ifvmode\leavevmode\fi\vadjust{
\vbox to0pt{\hbox to 0pt{\hskip\hsize\hskip1em
\vbox{\hsize3cm\tiny\raggedright\pretolerance10000
\noindent #1\hfill}\hss}\vbox to8pt{\vfil}\vss}}}

\newcounter{ssig}
\newcounter{ttig}
\begin{document}

\title[Continuation of $\mathbb K$-logarithm along curves]{On a continuation of quaternionic and octonionic logarithm along curves and the winding number}
\author{Graziano Gentili}
\address{DiMaI,     Universit\`a di Firenze, Viale Morgagni 67/A,\ Firenze, Italy}
\email { graziano.gentili@unifi.it }
\author{Jasna Prezelj}
\address{Fakulteta za matematiko in fiziko Jadranska 21 1000
  Ljubljana, Slovenija, UP FAMNIT, Glagolja\v ska 8, Koper Slovenija, IMFM, Jadranska 19, Slovenia}
\email { jasna.prezelj@fmf.uni-lj.si}
\author {Fabio  Vlacci}\address{DiSPeS Universit\`a di Trieste Piazzale Europa 1,\ Trieste,
  Italy} \email{ fvlacci@units.it} \thanks{\rm
    The first and third authors were partly supported by INdAM, through: GNSAGA; INdAM project ``Hypercomplex function theory and applications''.
  It was also partly supported by MIUR, through the projects: Finanziamento Premiale FOE 2014 ``Splines for accUrate NumeRics: adaptIve models for Simulation Environments''.
   The second author
  was partially supported by research program P1-0291 and by research
  project J1-3005 at Slovenian Research Agency. }

\maketitle
\section{Introduction}


This paper focuses on the problem
of finding a continuous extension of  the hypercomplex logarithm along  a path.
As pointed out  in \cite{GPV}, 
while a branch of the
complex logarithm can be defined in a small open neighborhood of a
strictly negative real point, no continuous branch of the hypercomplex  logarithm
can be defined in any open set $A\subset \bK\setminus \{0\}$ which
contains a strictly negative real point $x_0$ (here $\bK$ represents the algebra of quaternions or of octonions).

To overcome these difficulties, in \cite{GPV} we introduced the {\em logarithmic manifold}
$\mathscr
E_\bK^+$ and then showed that, if  $q\in\bK,\ q=x+Iy$
then  $E(x+Iy) 
= (\exp x \cos y + I\exp x \sin y, Iy)$
is an immersion and a diffeomorphism between
$\bK$ and $\mathscr E_\bK^+$.

In this paper, we consider lifts of paths in $\bK\setminus\{0\}$ to the logarithmic manifold
$\mathscr{E}^+_\bK$; even though $\bK \setminus \{0\}$ is simply
connected, in general, given a path in $\bK \setminus \{0\}$, the
existence of a lift of this path to $\mathscr{E}^+_\bK$ is not
guaranteed.  There is an obvious equivalence between the problem of
lifting a path in $\bK \setminus \{0\}$ and the one of finding a
continuation of the hypercomplex logarithm $\log_{\bK}$ along this path.

We want to recall that the slice regular logarithm $\log_*(f)$ of a slice regular function $f$ (see
\cite{AdF1,GPV1}) over the quaternions or octonions, introduced as the
slice regular inverse of the slice regular exponential $\exp_*(f)$ of
a slice regular function $f$ (see \cite{AdF}), is not defined in
general via the lift of $f$ to $\mathscr E_\bK^+$.
In particular it turns out that, in general,  $\log_*(f)(q)\neq \log_{\bK}(f(q))$.\\

The paper is organized as follows:
in Sections \ref{sec1} and \ref{sec2}, after recalling the basic
notions on slice regular exponential and logarithmic functions, we
provide explicit examples of paths intersecting the real axis and show
how a branch of the hypercomplex logarithm can be defined along
certain curves even when they encounter the real axis at negative
points, providing a so called \emph{continuation of the logarithm
  along a continuous curve}.

Furthermore, we introduce the notion of path and of loop {\em with
a  companion} (see Subsection \ref{sec31}) and then give a definition of
{\em winding number with respect to $0$} that has  a  full meaning for a
class of 
loops in $\mathbb K\setminus\{0\}\simeq
\mathbb{R}^{2^s}\setminus\{0\}\ (s=2,3)$ with companion; this fact is quite novel and original since it is
well known that a definition of winding number for a loop (with
respect to a point) is not in general possible  in  $\mathbb R^n$ when $n$
is greater than $2$.  Moreover this notion of winding number is
invariant  for the class of $c-$homotopic loops with companion.

Finally, in the last Section \ref{sec4}, we  extend the
continuation of the hypercomplex logarithm  to the case curves with
an infinite number of intersections with the real axis. 
These represent the set of obstructions for such an  extension.
When these obstructions are ``mild'' and ``reasonable'', then we also present an effective way to calculate the winding numbers using
the so-called notion of {\em signature}.

\section{Preliminary results}\label{sec1}

 We denote by $\mathbb{K}$  either the algebra of quaternions or octonions. Let
 $\mathbb{S}$ be the sphere of imaginary units, i.e. the set of
  $I \in \bK$ such that $I^2=-1$.  Given any  $z \in \bK \setminus
 \mathbb{R},$ there exist (and are uniquely determined) an imaginary
 unit $I$, and two real numbers $x,y$ (with $y> 0$) such that
 $z=x+Iy$. With this notation, the conjugate of $z$ will be $\bar z :=
 x-Iy$ and $|z|^2=z\bar z=\bar z z=x^2+y^2$.
 Each imaginary unit $I$ generates (as a real algebra) a copy
 of a complex plane denoted by $\mathbb{C}_I$. We call such a complex
 plane a {\em slice}.  The upper
half-plane in $\mathbb{C}_I$, namely $\{x+yI\ :\ y>0\}$ will be
 denoted by $\C_I^+$.
 Similarly, the lower half-plane in $\mathbb{C}_I$
   $\{x+yI\ :\ y<0\}$ will be denoted by
$\C_I^-$; each of these two half--planes will be
called a {\em leaf} of $\C_I$.
 On any leaf $\C_I^+$ we define
  the function $\arg_{I}:\C_I^+\to (0,\pi)$
 as $z=x+Iy\in\C_I^+ \mapsto \cot^{-1} (x/y):=\arg_{I}(z)$.

 The function $\arg_I$ can be continuously extended
 as a function $\arg_{I}:\C_I^+\cup\R^+\cup\R^-\to [0,\pi]$.

 It is also useful to define the imaginary unit
 function on $\bK \setminus \R$ in the following way:
 if $z \in\C_I^+,$ i.e. if $z=x+Iy$, with $x,y\in\mathbb{R}$ and $y>0$, then
 $\cI (z) = I$;
 if $z \in\C_I^-,$ i.e. if $z=x-Iy$, with $x,y\in\mathbb{R}$ and $y>0$, then
  $\cI (z) = -I$.


\begin{remark}\label{principaleasy}
\emph{
It is worthwhile noticing that the function $\cI$ cannot be extended as a continuous function to any single point of the
 real axis $\mathbb{R}$ of $\bK$. At the same time, if we set $\bS(-\pi, \pi)=\{Iy : I\in \bS, \ y\in (-\pi, \pi) \}$, then the function
 \[
 \Arg :\bK \setminus (-\infty, 0] \to \bS(-\pi, \pi)
 \]
 defined as the product
 \[
 \Arg (q):= \cI(q)\arg_{\cI}(q)
 \]
can be extended (as the zero function) to the positive real axis $\R^+$ of $\bK$.}
 \end{remark}

\section{The hypercomplex exponential and logarithm}\label{sec2}

\noindent Let us recall that the  exponential map on $\bK$ 
\[
\exp: \bK \to \bK\setminus\{0\}
\]
defined as
\[
\exp(q)=\sum_{k\geq 0} \dfrac{q^{k}}{k!}
\]
is a {\em slice regular} and {\em slice preserving} entire function on $\bK$ (\cite{AdF, GSS}).
Let $\mathscr E_\bK^+=f(\bK^+)$ denote the \emph{logarithm manifold}, i.e., the image $T(\bK^+)$ of $\bK^+=\{q \in \bK : \Re\ q >0\}$ of the map $T: \bK \to \bK\times \Im(\bK)$
defined by
\begin{eqnarray*}
T(x+Iy)=(\sinh x \cos y +I\sinh x \sin y, Iy)
\end{eqnarray*}
for $I\in \bS$, $x, y\in \R$.
The $\mathscr E_\bK^+$-\,exponential map
\[
E : \bK \to \mathscr E_\bK^+ \subset \bK \times \Im(\bK)
\]
defined by:
\[
E(x+Iy) = (\exp (x + Iy), Iy) = (\exp x \cos y + I\exp x \sin y, Iy)
\]
is an immersion and a diffeomorphism between $\bK$ and $\mathscr
E_\bK^+$ (see \cite{GPV}). In the case of quaternions, it endows
$\mathscr E_\H^+$ with a structure of slice quaternionic manifold
(see, e.g., \cite{GGS}), which is different from the structure of
hypercomplex Riemann manifold defined in Propostion 4.3. \cite{GPV})

The next definition and result appear in \cite{GPV}.

\begin{definition} Let $\mathscr E_\bK^+$ be the semi-helicoidal
 hypercomplex manifold.
The $\mathscr E_\bK^+$-\,logarithm
\[
L : \mathscr E_\bK^+ \subset \bK \times \Im(\bK)   \to \bK
\]
 is defined as follows, in terms of the real logarithm  $\log$:
\[
L (q, p) = \log |q| + p
\]
Indeed, if $(q,p)\in \ \mathscr E_\bK^+$, then $q=r\exp p$ for $r=|q|$ and our definition can be rewritten as:
\[
L (r\exp p, p) = \log r + p
\]
\end{definition}
The  hypercomplex manifold
$\mathscr E_\bK^+$ plays the role of an ``adapted" blow-up of $\bK$ at points of the form $x+2Ik\pi$, for $k \in \Z$ and $k\neq 0$.
\begin{proposition}
The map
\[
L : \mathscr E_\bK^+ \to \bK
\]
is the inverse of the $\mathscr E_\bK^+$-\,exponential $E$, and a diffeomorphism from the logarithm manifold $\mathscr E_\bK^+$ to $\bK$.
\end{proposition}

Note that if $\pro : \bK \times \Im(\bK) \to \bK$ denotes the projection on the first factor, then by definition the following equality holds
\[
\pro \circ E(q)= \exp(q)
\]
for all $q\in \bK$. Indeed, the map $L$ is a slice regular map from  $\mathscr E^+$ to $\H$, with respect to the structure of slice regular manifold induced by $E$ on $\mathscr E_\bK^+$ (see, e.g., \cite{GGS}). This map allows the definition of the hypercomplex logarithm (see \cite{GPV,GPV1}):

\begin{definition}\label{det1log}
Let
$
\pro :  \mathscr E_\bK^+  \subset \bK \times \Im(\bK)  \to \bK\setminus \{0\}
$
denote the natural projection
\[
(q, p) \mapsto q
\]
and let $\Omega \subset \mathscr E_\bK^+$ be a path connected subset such that $\pro_{|_{\Omega}}$ is injective. Then, the map
\[
\log_\bK : \pro(\Omega) \to \bK
\]
defined by
\[
\log_\bK q = L( \pro_{|_{\Omega}}^{-1}(q) )
\]
is called a \emph{branch} or \emph{a determination of  the hypercomplex logarithm} on $\pro(\Omega)$.
\end{definition}
As one can expect, it holds
\[
\exp (\log_\bK q)= \pro (E  (L( \pro_{|_{\Omega}}^{-1}(q) )))=\pro (\pro_{|_{\Omega}}^{-1}(q) )=q
\]
for all $q$ in $\pro(\Omega)$.
It is worthwhile noticing that,
if we consider the open, path-connected subset
\[
\Omega= \{(q, \Arg(q)) : q\in \bK\setminus (-\infty, \pi] \} \subset  \mathscr E_\bK^+ 
,\]
then the projection on the first factor
\[
\pro: \Omega \to \bK\setminus (-\infty,  0]
\]
is injective. Therefore, in this way, one defines
the {\em principal branch} of the logarithm (see \cite{GV})
in $\pro(\Omega)= \bK\setminus (-\infty, \pi]$ (see \cite{GPV1, AdF1}).
 The principal branch of the hypercomplex logarithm
 \[
  \log_{0} : \bK \setminus (-\infty, 0] \to \R  \times [0,\pi) \bS  \subset \bK
 \]
 \[
 q \mapsto \log|q| + \Arg (q)
 \]
 is well defined and, for all $I\in \bS$, coincides with the principal branch of the complex logarithm in the slice $\C_I \setminus  (-\infty, 0] $.
   As a consequence, $\log_{0}$ is a slice regular function in the symmetric slice domain $\bK \setminus  (-\infty, 0] $.

%
%

 As already observed in the Introduction, despite the analogy with the
 complex holomorphic case, in general no continuous branch of the
 hypercomplex logarithm can be defined in any open set $A\subset
 \bK\setminus \{0\}$ which contains a strictly negative real point
 $x_0$.  Nevertheless, we will now see how a branch of the
   hypercomplex logarithm can be defined along certain curves even
   when they encounter the real axis at negative points, providing a
   so called \emph{continuation of the logarithm along a continuous
     curve}.

Throughout the paper, a continuous curve will be called a \emph{ path}, and a closed path will be called a \emph{loop}.

  \begin{definition}\label{continulog}
Let $\gamma : [a,b] \ra \bK \setminus \{0\}$ be a path. Then a  path $\widetilde\gamma: [a,b] \ra \bK$ is called {\em a continuation of the logarithm along $\gamma$ } if $$\exp\circ\widetilde\gamma = \gamma,$$ i.e., if the following diagram commutes:
$$\xymatrix
  {&  &   \ar[d]^{\exp}  \bK   \\
    \ar[urr]^{\widetilde\gamma}
    [a,b]   \ar[rr]^{\gamma}  & & \bK \setminus\{0\}}
$$
The point $\widetilde\gamma(a)\in \bK$ will be called the initial point of the continuation $\widetilde \gamma$.
\end{definition}

To study the possible continuations of the logarithm along a path,  we need to specifically define the various branches of the hypercomplex argument of an element from $\bK \setminus \{0\}$.
\begin{definition}\label{arg_2l arg_2l+1}
If $k \in \Z$, for all $q\in  \bK \setminus \R$, $q=x+\cI y$ with $y>0$, let us define for 
\[
\begin{array}{ll}
  k=2l:  & \cI_{2l}(q) = \cI(q), \,\, \arg_{2l}(q) = \arg_{\cI}(q) + 2l\pi, \label{DefBranches}\\
  & \mbox{where }  \arg_{2l}(q)\in (2l\pi, (2l+1)\pi),\\
&  \\
  k=2l+1: &  \cI_{2l+1}(q) = -\cI(q),\,\, \arg_{2l+1}(q) = (2\pi -\arg_{\cI}(q))  + 2l\pi,\\
 & \mbox{where } \arg_{2l+1}(q)\in ((2l+1)\pi, (2l+2)\pi).
 \end{array}
\]
The $k$-th branch of the hypercomplex argument
$$\Arg_k: \bK \setminus \R \rightarrow \bS(k\pi,(k+1)\pi)$$
is defined by setting
\[
\Arg_k(q):= \cI_k(q)\arg_k(q).
\]
\end{definition}
As a consequence,
\begin{eqnarray}
\Arg_{2l+1}(q)&=&-\cI(q)(2\pi -\arg_{\cI}(q)+ 2l\pi)\\ \nonumber
&=& \cI(q)(\arg_{\cI}(q) -2(l+1)\pi)\\ \nonumber
&=& \cI(q)\arg_{-2(l+1)}(q)\\ \nonumber
&=& \Arg_{-(2l+2)}(q) \nonumber
\end{eqnarray}
Therefore, the only different branches of the hypercomplex argument of a quaternion $q\in  \bK \setminus \R$, $q=x+\cI y$ with $y>0$, can be listed for $k\in \Z$ as
\[
\Arg_{2k}(q):= \cI(q)\arg_{2k}(q).
\]
It is worthwhile noticing that for any fixed $q\in  \bK \setminus \R$, we have that for all $k\in \Z$
\[
\exp(\Arg_{2k}(q))=\exp(\Arg(q));
\]
indeed
\begin{eqnarray*}
\exp(\Arg_{2l}(q))&=&\exp[{\cI}(q)(\arg_{\cI}(q)+2l\pi)]=\exp[{{\cI}(q)\arg_{\cI}(q)}]\\
&=&\exp(\Arg(q)).
\end{eqnarray*}


%

\section{Continuation of hypercomplex logarithms along paths}\label{sec3}

The construction of a continuation of the logarithm along a path naturally involves the notion of  a  lift of a path.

\begin{definition}\label{LiftCurve}
Let $\gamma : [a,b] \ra \bK \setminus \{0\}$ be a path. Then a path
$\Gamma: [a,b] \ra \mathcal{E}_\bK^+$ is
{\em a lift of $\gamma$ (to  $\mathcal{E}_\bK^+$)} if
$\pro\circ \Gamma = \gamma$, i.e., if the
the following diagram commutes:
$$\xymatrix
  {&  &   \ar[d]^{\pro}  \mathcal{E}_\bK^+   \\
    \ar[urr]^{\Gamma}
    [a,b]   \ar[rr]^{\gamma}  & & \bK \setminus\{0\}}
$$
For $(q,p) \in \mathcal{E}_\bK^+$, a lift $\Gamma$ of $\gamma$ such that $\Gamma (a)=(q,p)$ will be said to have \emph{initial point} $(q,p)$.
\end{definition}
The existence of a lift of a path $\gamma$ is equivalent to the existence of a continuation of the hypercomplex logarithm along it.\\

\begin{proposition}\label{first lemma}
Let  $\gamma : [a,b]\subset\bR \ra \bK \setminus \{0\}$ be a  path.
Then, there exists a lift of $\gamma$ to $\mathcal{E}_\bK^+$ if, and only if, there exists a continuation  of the logarithm along $\gamma$.
\end{proposition}
\begin{proof} Suppose that there exists  a continuation of the logarithm $\widetilde \gamma$ along $\gamma$. Then
the path $\Gamma$ defined by $\Gamma(t) = ((\exp\circ \widetilde\gamma)(t), \Im(\widetilde\gamma(t))$ is obviously a lift of $\gamma$ to $\mathcal{E}_\bK^+$.

\vskip .2cm
$$\xymatrix
  {&  &   \ar[d]^{\pro}  \mathcal{E}_\bK^+  \ar[rr]^{L} & & \bK  \ar[lld]^{\exp} \\
    \ar[urr]^{\Gamma}
    [a,b]   \ar[rr]^{\gamma}  & & \bK \setminus\{0\}}
$$
\vskip .2cm
Conversely, if a lift $\Gamma$ of the path $\gamma: [a,b] \to \bK \setminus \{0\}$ to $\mathcal{E}^+_\bK$
exists, then a continuation of the hypercomplex logarithm along $\gamma$ can be defined by
$\widetilde\gamma(t):= L(\Gamma(t))$.
\end{proof}

Thanks to the result just stated, we are left to find conditions under
which a path $\gamma: [a,b] \to \bK \setminus \{0\}$ can be lifted to
$\mathcal{E}^+_\bK$. Since the map $\pro: \mathcal{E}_\bK^+ \to
\bK\setminus\{0\}$ is not a covering, we have to specifically
study the existence of lifts of $\gamma$.

Let us first consider the easy cases: it is not difficult to see that
if we restrict the map $\pro: \mathcal{E}_\bK^+ \to \bK\setminus\{0\}$
to the preimage of $\bK\setminus \R$, then the restriction
$\pro_{|\pro^{-1}(\bK\setminus \R)}$ becomes a covering. Indeed it
becomes a trivial covering, since $\pro^{-1}(\bK\setminus \R)$ is
homeomorphic (namely diffeomorphic) through the diffeomorphism
\[
E : \bK \to \mathscr E_\bK^+ 
\]
to the countable collection of open simply connected domains given by
$$
\bK \setminus \left\{\bigcup_{k\in \Z}\R\times \bS k\pi\right\}=\bigcup_{k\in \Z} \R\times \bS(k\pi, (k+1)\pi).
$$
Let us now set, for any $k\in \Z$,
\[
D_k= \R\times \bS(k\pi, (k+1)\pi), \qquad E(D_k)= E(\R\times \bS(k\pi, (k+1)\pi))\subset \mathscr E_\bK^+ .
\]
 Notice that $\bS(k\pi, (k+1)\pi)=\bS(-(k+1)\pi, -k\pi))$ and hence $D_{2k} = D_{-(2k+1)}.$
With this in mind, we can now state the following proposition.

\begin{proposition}\label{easylifts}
Assume the path $\gamma:[a,b] \ra \bK \setminus \{0\}$ is
such that $\gamma([a,b])\cap \mathbb{R}=\varnothing$, and let
\[
\gamma(t) = x(t)+{\cI}(t)y(t)
\]
with $y(t)>0$ for all $t\in [a,b]$. Then, for any $k\in\Z$, there exists one, and only one, lift $\Gamma_{k}$ of $\gamma$ to $E(D_{2k})\subset \mathscr E_\bK^+$ with initial point
$$\Gamma_{k}(a)=(\gamma(a), \Arg_{2k}(\gamma(a))).$$
Namely, for all $t\in [a,b]$, we have
\[
\Gamma_{k}(t) := (\gamma(t), \Arg_{2k}(\gamma(t))) \in E(D_{2k}).
\]
Finally, for all $k\in \Z$, the map defined on the interval $[a,b]$ by
\begin{equation}\label{-kth branch on K-R}
(\log_{k}\circ \gamma)(t) := (L\circ\Gamma_{k})(t) = \log|\gamma(t)| + \Arg_{2k}(\gamma(t))
\end{equation}
is the unique continuation of the hypercomplex logarithm along $\gamma$ with initial point $\log|\gamma(a)| + \Arg_{2k}(\gamma(a))$, and is called  \emph{the $k$-th branch of the hypercomplex logarithm along $\gamma$ with initial point $\log|\gamma(a)| + \Arg_{2k}(\gamma(a))$.}
\end{proposition}

\begin{proof}
For each $k\in \Z$, the proof of the existence and uniqueness of $\Gamma_k$ as in the statement is a straightforward consequence of what already established. To prove the last part of the statement, let us consider the graph $\Omega_k$ of  the lift $\Gamma_{k}$ of $\gamma$ to $E(D_{2k})\subset \mathscr E_\bK^+$ with initial point $\Gamma_{k}(a)=(\gamma(a), \Arg_{2k}(\gamma(a)))$, i.e.,
 \[
   \Omega_k:= \{(q,\Arg_{2k}(q)) : q \in \gamma([a,b])\} \subset E(D_{2k})\subset  \mathcal{E}^+_\bK.
\]
Since the projection  $\pi: \mathcal{E}^+_ \bK \to \bK \setminus \{0\}$ restricted to $\Omega_k$   is injective, following Definition \ref{det1log}, we obtain \eqref{-kth branch on K-R}.
\end{proof}

 Under the hypotheses of the preceding proposition, loops lift to loops, hence:
\begin{corollary}
Assume the loop $\gamma:[a,b] \ra \bK \setminus \{0\}$ is
such that $\gamma([a,b])\cap \mathbb{R}=\varnothing$. Then for each $k\in \Z$, the lift $\Gamma_k$ found in Proposition \ref{easylifts} is a loop. As a consequence, for each $k\in \Z$,
$$\log_k(\gamma(a))=\log_k(\gamma(b))$$
\end{corollary}

Among the initial cases, there is the one corresponding to what is stated in Remark \ref{principaleasy}.

\begin{proposition}\label{k meno R meno}
Assume the path $\gamma:[a,b] \ra \bK \setminus \{0\}$ is
such that $\gamma([a,b])\cap \mathbb{R}^-=\varnothing$.  Then there
exists a lift $\Gamma$ of $\gamma$ to $\mathcal{E}^+_\bK$.
\end{proposition}
\begin{proof}
The proof is a consequence of what observed in Remark \ref{principaleasy}. Indeed, taking into account that $\bS(-\pi, \pi)=\bS[0, \pi)=\bS(-\pi, 0]$, the mentioned remark implies that $\pi:E( \R \times \bS(-\pi, \pi)) \to  \bK \setminus (-\infty, 0]$ is a homeomeorphism.
\end{proof}

As pointed out in the Introduction, even though $\bK \setminus \{0\}$ is simply connected, in general, given a  path in  $\bK \setminus \{0\}$, the existence of a  lift of this path to $\mathcal{E}^+_\bK$ is not guaranteed. Indeed, consider the following examples.
\begin{example}\label{esempio1}
\emph{\begin{enumerate}[(a)]
\item Let $\sigma:[\pi/2, 3\pi/2] \to \bK \setminus \{0\}$ be the path (depicted in Figure \ref{GArc}) defined by
$$\sigma(t)=\cos(t)+I(t)\sin(t)$$
where $I: [\pi/2, 3\pi/2] \to \bS$ is defined as
\[
\,\,\,\,\,I(t)=i \quad \textnormal{for}\quad  \pi/2\leq t < \pi \quad  \textnormal{and } \quad  I(t)=-j \quad  \textnormal{for}\quad \pi\leq t \leq 3\pi/2.
\]
The curve $\sigma$ is continuous, but the function
$$\Arg(\sigma(t))=I(t)\arg_I(\sigma(t))$$
is not continuous at $\pi$ (the left and right limits are different). Therefore $\sigma$ cannot be lifted to $\mathcal{E}^+_\bK$.
\begin{figure}[h!]
\centering
 \includegraphics[width=0.45\textwidth]{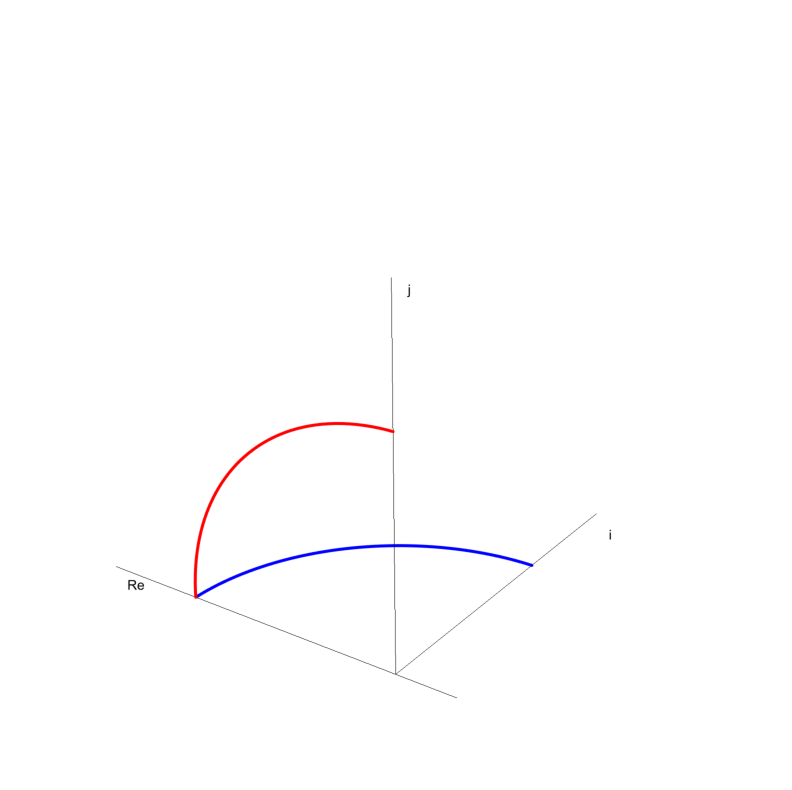}
 \begin{center}\caption{\footnotesize The arc $\sigma\qquad\qquad\qquad$  }
 \label{GArc}
 \end{center}
\end{figure}
\\
\item Consider now the loop $\gamma:[0,1] \to \bK \setminus \{0\}$ defined by
      $$
  \gamma(t) = \cos(\pi-2 \pi t) +  t(1-t)(i \cos(2\pi/t) + j \sin(2\pi /t)),
    $$
where $i,j$ are the usual orthogonal imaginary units (see Figure \ref{ri3p}).  Notice that the imaginary part of $\gamma$ is continuous at all points of the interval $[0,1]$ (including $0$). Nevertheless, for $t$ near to $0$, we have that the function
\begin{eqnarray*}
 &&\Arg(\gamma(t))=\\
 &&=(i \cos(2\pi/t) + j \sin(2\pi /t))\arccos\left(\frac{\cos(\pi-2 \pi t)}{\sqrt{\cos^2(\pi-2 \pi t)+t^2(1-t)^2}}\right)
\end{eqnarray*}
has no limit for $t$ approaching $0^+$. Therefore $\gamma$ cannot be lifted to $\mathcal{E}^+_\bK$.\\
\item Notice that in both the preceding cases, the paths  $\hat\sigma:=-\overline{\sigma}$ and $\hat\gamma:=-\overline{\gamma}$ can be lifted to $\mathcal{E}^+_\bK$, since their images are included in $\bK \setminus (-\infty, 0]$ (see Proposition \ref{k meno R meno}).
\end{enumerate}
 }
\end{example}
 \begin{figure}[h!]
\begin{center}
 {\includegraphics[width=0.45\textwidth]{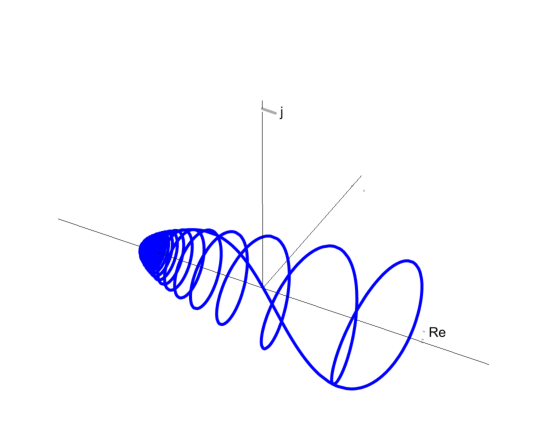}}\quad
 {\includegraphics[width=0.45\textwidth]{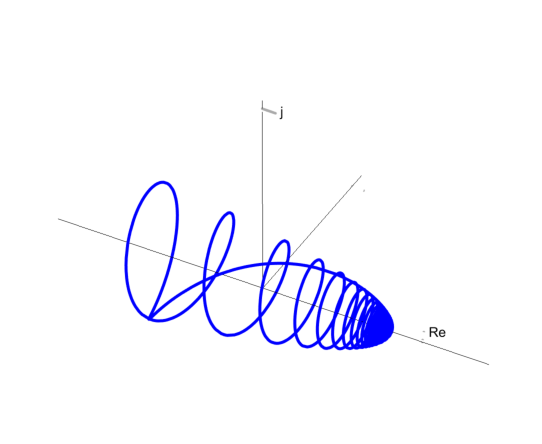}}
  \end{center}
 \caption{\footnotesize The path $\gamma$ (negative rocket) of the Example \ref{esempio1} (b) is drawn on the left: it cannot be lifted to $\mathcal{E}^+_\bK$. Its reflection on the right (positive rocket) can be lifted to $\mathcal{E}^+_\bK$. }
 \label{ri3p}
\end{figure}

It is useful to point out that the existence of a lift $\Gamma$ of a path $\gamma$ to $\mathcal{E}_\bK^+$
is equivalent to the existence of a continuous function
$\Arg^{\gamma}: [a,b] \ra \Im(\H)$, such that
\[ \Gamma(t) = (\gamma(t), \Arg^{\gamma}(t)) \in \mathcal{E}^+_\bK.\]
As noticed in Remark \ref{principaleasy}, the function $\Arg^{\gamma}$ will be decomposed, where possible, with obvious notation, as
\[\Arg^{\gamma}:= \cI^{\gamma}\arg^{\gamma} \]
where $\cI^{\gamma}:[a,b] \ra \bS$ and $\arg^{\gamma}: [a,b] \ra \R$.  The existence of
$\cI^{\gamma}:[a,b] \ra \bS$ implies that we can assign to each
$t\in [a,b]$ a complex plane $\C_{\cI^{\gamma}}$ which contains the
point $\gamma(t)$ and hence determines the argument up to a multiple of
$2\pi.$

Complex slices $\{\C_J\}_{J\in\mathbb{S}}$ are naturally parameterized by the elements of $\bS/ \{\pm \Id\}$, the real projective space $\R \P^{\dim_{\R}\bK - 2}$ of dimension $\dim \bK -2$.
The projection $[\ ]: \bS \ra \bS/ \{\pm \Id\}=\R
\P^{\dim_{\R}\bK - 2}$ is the classical $2:1$ universal covering map and, as customary, for $J\in \bS$, the symbol $[J]$ denotes the equivalence
class whose representatives are the opposite (conjugate) imaginary
units $J, -J\in \bS.$ Each element $[J] \in \bS/ \{\pm \Id\}$ uniquely defines the
complex slice $\C_{[J]} =\C_J=\C_{-J}$. A
continuous imaginary unit function $\cI^{\gamma}: [a,b] \to \bS$ naturally defines a continuous function $\mathfrak{I}^{\gamma}: [a,b] \to \bS/ \{\pm \Id\}$ when we set $\mathfrak{I}^{\gamma}(t) =
[\cI^{\gamma}(t)].$

\begin{definition}.
Let  $[a,b] \subset \bR$ and let $\gamma : [a,b] \ra \bK \setminus \{0\}$ be a  path.

A path  $\mathfrak{I}^{\gamma}:[a,b] \ra \bS/ \{\pm \Id\}$ such that $\gamma(t) \in
  \C_{\mathfrak{I}^{\gamma}(t)}$ for every $t\in[a,b]$ is called a \emph{companion} of the path $\gamma$.

  If a companion $\mathfrak{I}^{\gamma}$ of the path $\gamma$ exists, then $\gamma$ is called a  \emph{path with a companion} and the pair $(\gamma, \mathfrak{I}^{\gamma})$ is called a \emph{path with companion} $\mathfrak{I}^{\gamma}$.

  If the path $\gamma$ has a unique companion $\mathfrak{I}^{\gamma}$, then both $\gamma$ and the pair $(\gamma, \mathfrak{I}^{\gamma})$ are called a \emph{tame} path.

\begin{proposition}\label{realization}
Let $\gamma: [a, b] \to \mathbb K\setminus \{0\}$ be a path with companion $\mathfrak{I}^{\gamma}$. If $\cI^{\gamma}, - \cI^{\gamma}: [a, b] \to \bS$ are the two lifts of of $\mathfrak{I}^{\gamma}$, then there exist continuous functions $x,y : [a, b] \to \mathbb R$ such that, for all $t\in [a, b]$,
\[
\gamma(t)=x(t)+\cI^{\gamma}(t)y(t)=x(t)+(-\cI^{\gamma}(t))(-y(t)).
\]
These last expressions are called \emph{canonical forms of} $(\gamma, \mathfrak{I}^{\gamma})$.
\end{proposition}
\begin{proof}
Since $\gamma(t)$ and $\cI^{\gamma}(t)$ are both continuous, then $\Re(\gamma(t))=x(t)$ and $-\cI^{\gamma}(t)\Im(\gamma(t))=y(t)$ are continuous as well on $[a, b]$.
\end{proof}


%
\end{definition}

It is easy to see that all paths lying entirely in a complex slice have a companion.  Notice as well that a  path $\gamma: [a,b] \ra \bK \setminus \{0\}$ may have more than one companion: this happens for example when the path $\gamma$ is  such that  
 $\gamma([a,b])\subset (0,\infty)$; in this case, for an arbitrary path $\cI^{\gamma}:[a,b] \to \bS$,  the induced path $\mathfrak{I}^{\gamma}: [a,b] \ra \bS/ \{\pm \Id\}$ is  a companion of $\gamma$; consequently $\gamma$ is not tame. For a similar reason, a path   $\gamma: [a,b] \ra \bK \setminus \{0\}$ which maps a closed sub-interval of $[a,b]$ to a real number has more than one companion, and hence is not tame.

\begin{remark}
\emph{There exist paths in  $\bK \setminus \{0\}$ which can be lifted to $\mathcal{E}^+_\bK$, but have no companion. Indeed, set
\[
\hat{\sigma}=-\overline \sigma :[\pi/2, 3\pi/2] \to \bK \setminus \{0\}
\]
where $\sigma$ is the path defined in Example \ref{esempio1} (a). The path $\hat{\sigma}$ is the symmetric  image of the path $\sigma$ with respect to the plane of purely imaginary quaternions (see figure \ref{GArc}) and, as pointed out in Example \ref{esempio1} (c), it can be lifted to  $\mathcal{E}^+_\bK$. Obviously $\hat{\sigma}$ has no companion: the continuity of a companion cannot hold at $t=\pi$.
}
\end{remark}

The following definition will play a central role in the sequel.

\begin{definition}\label{winding}  Let $\gamma: [a,b] \ra \bK \setminus \{0\}$
  be a path with a companion $\mathfrak{I}^{\gamma}: [a,b] \to \bS/\{\pm \Id \}$, let $\cI^{\gamma}, -{\cI^{\gamma}}:[a,b] \to \bS$ be the two h  (continuous) lifts of $\mathfrak{I}^{\gamma}$ to $\bS$ and let
 \[
  \gamma(t)=x(t)+\cI^{\gamma}(t)y(t)=x(t)+(-\cI^{\gamma})(t)(-y(t))
  \]
  be the canonical forms of $(\gamma, \mathfrak{I}^{\gamma})$.
 The paths $\gamma_{\cI^{\gamma}}, \gamma_{-\cI^{\gamma}}: [a,b] \to \C \setminus \{0\}$ defined by
  \[
 \gamma_{\cI^{\gamma}}=x(t)+iy(t), \qquad \qquad \gamma_{-\cI^{\gamma}}=x(t)-iy(t)
  \]
are called   \emph{ the (two conjugated) shadows associated with the pair  $(\gamma,\mathfrak{I}^{\gamma})$}. 
 If the path $\gamma$ is tame, then the paths $\gamma_{\cI^{\gamma}}$ and $ \gamma_{-\cI^{\gamma}}$ are simply called \emph{the (two) shadows associated with the path} $\gamma$.
\end{definition}

\begin{remark}\emph{
The two shadows associated with the pair ($\gamma$, $\mathfrak{I}^{\gamma}$) are conjugate paths.}
\end{remark}


Paths with a companion are of interest because they can all be lifted to $\mathcal{E}^+_\bK$.

\begin{proposition}\label{TameLift1} Let $\gamma: [a,b] \ra \bK \setminus \{0\}$ be a path with companion $\mathfrak{I}^{\gamma}:[a,b] \ra \bS/ \{\pm \Id\}$. Then there exist
\begin{itemize}
\item   a path $\cI^{\gamma}:[a,b] \to \bS$   with $[ \cI^{\gamma}(t)]=\mathfrak{I}^{\gamma}(t)$, for all $t\in [a,b]$,
\item a path $\arg^{\gamma}:[a,b] \to \R$,
\end{itemize}
such that, after setting $\Arg^{\gamma} = \cI^{\gamma} \arg^{\gamma}: [a,b] \to \Im(\H)$, the path
\[\Gamma(t) = (\gamma(t), \Arg^{\gamma}(t)) \]
is a lift of $\gamma$ to $\mathcal{E}^+_\bK$ with $\arg^{\gamma}(a)\in
[0,\pi]$, called a \emph{$\mathfrak{I}^{\gamma}$-lift of $\gamma$}.

\noindent If, as in Definition \ref{arg_2l arg_2l+1},  for every $k\in \Z$ we set $\arg^{\gamma}_{2k}:= \arg^{\gamma} + 2k \pi$ and $\Arg_{2k}^\gamma := \cI^{\gamma} \arg_{2k}^{\gamma}$, then  the path
\[
\Gamma_{k}(t)=(\gamma(t), \Arg_{2k}^{\gamma}(t))
\]
is a  $\mathfrak{I}^{\gamma}$-lift of $\gamma$ to $\mathcal{E}^+_\bK$ with $\arg^{\gamma}_{2k}(a)\in
[2k\pi, (2k+1)\pi]$.
%
%
\end{proposition}

\begin{proof}
There exist exactly two continuous lifts
$\cI^{\gamma}, -\cI^{\gamma}$ of $\mathfrak{I}^{\gamma}$ to the universal covering $\bS$ of $\bS/ \{\pm \Id\}$. Correspondingly, there exist two shadows $\gamma_{\cI^{\gamma}}, \gamma_{-\cI^{\gamma}}: [a,b] \to \C \setminus \{0\}$ associated with $\mathfrak{I}^{\gamma}$. Exchange  $\cI^{\gamma}$ and $-\cI^{\gamma}$ if necessary, so that $\cI^{\gamma}$ is such that 
$\arg(\gamma_{\cI^{\gamma}}(a))\in [0,\pi]$.
As a complex path, $\gamma_{\cI^{\gamma}}$ has a well defined argument $\arg^{\gamma_{\cI^{\gamma}}}:[a,b] \to \R$ such that $\arg^{\gamma_{\cI^{\gamma}}}(a)\in [0, \pi]$. Set $\arg^{\gamma} := \arg^{\gamma_{\cI^{\gamma}}}$. Then the chosen paths  $\cI^{\gamma}$ and $\arg^{\gamma}$ have the properties required in the statement.
The rest of the proof is straightforward.
\end{proof}

The lifts $\Gamma$ and $\Gamma_k$ (for $k\in \Z$) appearing in the last Proposition are not unique, when $\Gamma(a)$ and $\Gamma_k(a)$ are real.

At this point, Proposition \ref{first lemma} implies directly the existence of all branches of the logarithm, along all paths in  $ \bK \setminus \{0\}$ having a companion.

\begin{corollary}

Let $\gamma: [a,b] \ra \bK \setminus \{0\}$ be a path with companion $\mathfrak{I}^{\gamma}:[a,b] \ra \bS/ \{\pm \Id\}$. For every $k\in \Z$, let
\[
\Gamma_{k}=(\gamma, \Arg_{2k}^{\gamma})
\]
be a  $\mathfrak{I}^{\gamma}$-lift of $\gamma$ to $\mathcal{E}^+_\bK$ with $\arg^{\gamma}_{2k}(a)\in
[2k\pi, (2k+1)\pi]$. Then, the map defined on the interval $[a,b]$ by
\begin{equation*}\label{-kth branch on K-0}
(\log_{k}\circ \gamma)(t) := (L\circ\Gamma_{k})(t) = \log|\gamma(t)| + \Arg_{2k}(\gamma(t))
\end{equation*}
is a continuation of the hypercomplex logarithm along $\gamma$ with initial point $\log|\gamma(a)| + \Arg_{2k}(\gamma(a))$. This map is called  \emph{a $k$-th branch of the hypercomplex logarithm along $\gamma$ with initial point $\log|\gamma(a)| + \Arg_{2k}(\gamma(a))$.}
\end{corollary}
\begin{proof}
The proof is a straightforward consequence of Proposition \ref{TameLift1} and Proposition \ref{first lemma}.
\end{proof}

We will now turn our attention to the case of loops of $\bK \setminus \{0\}$.

%
%
%
\begin{definition}
Let  $[a,b] \subset \bR$ and let  $\gamma : [a,b] \ra \bK \setminus \{0\}$ be a path with a companion $\mathfrak{I}^{\gamma}: [a,b] \ra \bS/\{\pm \Id\}$.

If both $\gamma$ and $\mathfrak{I}^{\gamma}$ are closed, then the path $\gamma$ is called \emph{a loop with companion} $\mathfrak{I}^{\gamma}$, and the pair $(\gamma,\mathfrak{I}^{\gamma}) $ is called \emph{a loop with companion}.

The  loop with companion $(\gamma,\mathfrak{I}^{\gamma})$ is called \emph{untwisted}  if $\mathfrak{I}^{\gamma}$ is homotopic to a constant in $\bS/\{\pm \Id \}$; if instead $\mathfrak{I}^{\gamma}$ is not homotopic to a constant, then $(\gamma,\mathfrak{I}^{\gamma})$ is said to be \emph{twisted}.

\end{definition}

In the most relevant case in which $\gamma$ is tame, we can specialize the definition as follows.
\begin{definition}
Let  $[a,b] \subset \bR$ and let  $\gamma : [a,b] \ra \bK \setminus \{0\}$ be a tame path with  companion $\mathfrak{I}^{\gamma}: [a,b] \ra \bS/\{\pm \Id\}$.

If both $\gamma$ and $\mathfrak{I}^{\gamma}$ are closed, then $\gamma$ is called \emph{a tame loop (with companion} $\mathfrak{I}^{\gamma}$), and the pair $(\gamma,\mathfrak{I}^{\gamma}) $  is called \emph{a tame loop}.

The tame loop $(\gamma,\mathfrak{I}^{\gamma})$  is called \emph{untwisted}  if $\mathfrak{I}^{\gamma}$ is homotopic to a constant in $\bS/\{\pm \Id \}$; if instead $\mathfrak{I}^{\gamma}$ is not homotopic to a constant, then $(\gamma,\mathfrak{I}^{\gamma})$ is said to be \emph{twisted}.
\end{definition}

\begin{remark}
   {\em For any fixed $I\in \bS$, let $\gamma:[a,b] \to \bK \setminus \{0\}$ be a path lying in the complex slice $\C_I$. The path $\gamma$  has always a particularly simple companion, namely $\mathfrak{I}^{\gamma}:[a,b] \to \bS/\{\pm \Id \}$ constantly equal to $[I]$. Moreover, the two different lifts of  $\mathfrak{I}^{\gamma}$ to $\bS$ are both constantly equal to $I$ or $-I$, respectively. As a consequence, if the given path  $\gamma$ is closed and tame, it is a tame loop and is untwisted.}
\end{remark}

A twisted loop necessarily intersects the real axis. Indeed the following result holds.

\begin{proposition}\label{real axis}  Let  $\gamma: [a,b] \ra \bK \setminus \{0\}$ be a loop which misses the real axis. Then $\gamma$ is  a tame loop and is untwisted. 
\end{proposition}
\begin{proof}
By Proposition \ref{k meno R meno}, the loop $\gamma$ can be lifted to a path $\Gamma : [a,b] \to  \mathcal{E}^+_\bK$ with $\Gamma = (\gamma, \Arg^{\gamma})$. Let us consider the map $\Arg^{\gamma} = \cI^{\gamma} \arg^{\gamma}: [a,b] \to \Im(\bK)$. By the hypothesis, there exists $k\in \Z$ such that the map  $\arg^{\gamma}=\arg^{\gamma}_{2k}: [a,b] \to (2k\pi, (2k+1)\pi)$ is never vanishing and hence has constant sign. Now, since $\gamma$ is closed, we have that 
\[
\cI^{\gamma}(a) \arg^{\gamma}_{2k}(a)=\cI^{\gamma}(b)\arg^{\gamma}_{2k}(b).
\]
Since $\arg^{\gamma}_{2k}(a)$ and $\arg^{\gamma}_{2k}(b)$ have the same sign and both belong to the interval $(2k\pi, (2k+1)\pi)$, we obtain
\[
\arg^{\gamma}_{2k}(a)=\arg^{\gamma}_{2k}(b)
\]
and hence
\[
\cI^{\gamma}(a) =\cI^{\gamma}(b).
\]
Therefore the path $\cI^{\gamma}: [a,b] \to \bS$ is a loop, and hence the unique companion $[\cI^{\gamma}]:[a,b] \to \bS/\{\pm \Id \}$ is a loop,  homotopic to a constant. As a consequence the path $\gamma$ is an untwisted, tame loop.
\end{proof}

\subsection{Winding number for untwisted loops with companion in $\bK\setminus\{0\}$}\label{sec31}
It is well known that the definition of winding number for a loop (with respect to a point) is not natural in  $\mathbb R^n$ when $n$ is greater than $2$. Nevertheless, in our setting, we can start by giving a definition of winding number that has full meaning for 
loops with companion that are untwisted and lie in $\mathbb K\setminus\{0\}$.

The following result opens a way to this definition of winding number.

\begin{proposition}\label{untwisted loop}
A loop $\gamma: [a,b] \ra \bK \setminus \{0\},$ $\gamma([a,b]) \not\subset \R$, with companion $\mathfrak{I}^{\gamma}$ is untwisted if, and only if, for any chosen non real initial point of $\gamma$, both shadows associated with $\mathfrak{I}^{\gamma}$ are loops.
\end{proposition}
\begin{proof} 
 Let $\gamma_{\cI^{\gamma}}: [a,b] \to \C \setminus \{0\}$,
 $
 \gamma_{\cI^{\gamma}}(t)=x(t)+iy(t),
$
be one of the shadows associated with $\mathfrak{I}^{\gamma}$.

If the loop $\gamma$ is untwisted, then any lift $\cI^{\gamma}$ of the companion $\mathfrak{I}^{\gamma}$ of $\gamma$ is a loop, and hence it has coinciding endpoints. Therefore, the path
  \[
  \gamma(t)=x(t)+\cI^{\gamma}(t)y(t)
  \]
being a loop, the continuous function $y: [a,b]\to \R$ is such that $y(a)=y(b)$. Hence the associated shadow  $
 \gamma_{\cI^{\gamma}}(t)=x(t)+iy(t)
$
is closed.

 On the other hand, suppose the associated shadow $\gamma_{\cI^{\gamma}}: [a,b] \to \C \setminus \{0\}$,
 $
 \gamma_{\cI^{\gamma}}(t)=x(t)+iy(t)
$,
is a loop  and assume that  $y(a)=y(b) \ne 0.$ Since  the path
  \[
  \gamma(t)=x(t)+\cI^{\gamma}(t)y(t)
  \]
 is a loop by assumption, we obtain $\cI^{\gamma}(a)=\cI^{\gamma}(b)$ and so the lift $\cI^{\gamma}$ of $\mathfrak{I}^{\gamma}$ is  a loop. In conclusion, $\gamma$ is untwisted.
\end{proof}

We are now ready to use the well established definition of winding number for complex loops in $\C\setminus\{0\}$ to define the winding number in the case of untwisted loops with companion in $\bK\setminus\{0\}$.

\begin{definition} \label{windingnew}  Let the loop $\gamma: [a,b] \ra \bK \setminus \{0\}$ with companion $\mathfrak{I}^{\gamma}$ be untwisted.
The {\em winding number} (with respect to zero) of the loop $(\gamma, \mathfrak{I}^{\gamma})$, denoted $\emph{wind}(\gamma, \mathfrak{I}^{\gamma})$, is defined as the absolute value of the winding number (with respect to zero), $\emph{wind}(\gamma_{\cI^{\gamma}})$, of a shadow $\gamma_{\cI^{\gamma}}$ associated with $\mathfrak{I}^{\gamma}$:
\[
\emph{wind}(\gamma, \mathfrak{I}^{\gamma})=|\emph{wind}(\gamma_{\cI^{\gamma}})|.
\]
In the case in which the loop $(\gamma, \mathfrak{I}^{\gamma})$ is tame, there is one and only one companion of $\gamma$, and hence we can simply denote the winding number of $\gamma$ by
$\emph{wind}(\gamma)$.
\end{definition}

Of course, we need to show that the given definition of winding number of an untwisted loop with companion $(\gamma, \mathfrak{I}^{\gamma})$ does not depend on the choice of the shadow associated with $\mathfrak{I}^{\gamma}$. Indeed, the two shadows associated with $\mathfrak{I}^{\gamma}$ are conjugate loops: as a consequence, their winding numbers are opposite. Therefore, Definition \ref{windingnew} is consistent.

One of the important features of the classical winding number (with
respect to zero) of loops of $\C\setminus\{0\}$ is its invariance with
respect to homotopy between such loops. The winding number of an
untwisted loop with companion (in $\bK\setminus\{0\}$) just defined
cannot be invariant with respect to standard homotopy in $\bK\setminus\{0\}$:
all such loops are homotopic to a constant loop since
$\bK\setminus\{0\}$ is simply connected, and a constant loop has
vanishing winding number.

A special notion of homotopy comes into the scenery in our setting. The next definition is useful to define such a notion.

\begin{definition}
Let  $[a,b]\times [c,d] \subset \bR^2$ and let $F : [a,b]\times [c,d]  \ra \bK \setminus \{0\}$ be a  continuous map.

A continuous map  $\mathfrak{I}^{F}:[a,b]\times [c,d]  \ra \bS/ \{\pm \Id\}$ such that $F(t,s) \in
  \C_{\mathfrak{I}^{F}(t,s)}$ for every $(t,s)\in[a,b]\times [c,d] $ is called a \emph{companion} of the map $F$.

  If a companion $\mathfrak{I}^{F}$ of the map $F$ exists,  then $F$  is called a \emph{continuous map with companion} $\mathfrak{I}^{F}$, and $(F, \mathfrak{I}^{F})$ is called a \emph{continuous map with companion.}

  If the map $F$ has a unique companion $\mathfrak{I}^{F}$, then it is called a \emph{tame} map.
  \end{definition}

\begin{proposition}\label{realization2}
Let $F: [a, b]\times [c, d] \to \mathbb K\setminus \{0\}$ be a continuous map with companion $\mathfrak{I}^{F}$. If $\cI^{F}, - \cI^{F}: [a, b]\times [c, d]  \to \bS$ are the two lifts of $\mathfrak{I}^{F}$, then there exist continuous functions $x,y : [a, b]\times [c, d] \to \mathbb R$ such that, for all $(t, s)\in [a, b]\times [c, d]$,
\[
F(t)=x(t,s)+\cI^{F}(t, s)y(t, s)=x(t,s)+(-\cI^{F}(t,s))(-y(t,s)).
\]
These last expressions are called \emph{canonical forms of} $(F, \mathfrak{I}^{F})$.
\end{proposition}
\begin{proof}
See the proof of Proposition \ref{realization}
\end{proof}

As announced, the idea is now to define a special type of homotopy between paths, each having a companion and sharing the same endpoints. As customary, also in this paper homotopy between paths will always be meant with fixed endpoints.

\begin{definition} \label{c-homotopy}
Let $\gamma_1, \gamma_2 : [a,b]  \ra \bK \setminus \{0\}$ be two paths with the same endpoints $\gamma_1(a)=\gamma_2(a)=p$
and $\gamma_1(b)=\gamma_2(b)=q$. Let  $\mathfrak{I}^{\gamma_1}$ and $\mathfrak{I}^{\gamma_2}$ be companions of  $\gamma_1$ and $ \gamma_2$ respectively
.
If there exists a continuous map 
$F : [a,b]\times [0,1]  \ra \bK \setminus \{0\}$  with companion $ \mathfrak{I}^F$ such that:
\begin{enumerate}
\item $\mathfrak{I}^F(t,0)=\mathfrak{I}^{\gamma_1}(t)$ and $\mathfrak{I}^F(t,1)=\mathfrak{I}^{\gamma_2}(t)$, for all $t\in [a,b]$;
\item $F(t,0)=\gamma_1(t)$ and $F(t, 1)=\gamma_2(t)$, for all $t\in [a,b]$;
\item $F(0,s)=p$ and $F(1,s)=q$, for all $s\in [0,1]$;
\end{enumerate}
then we will say that $(\gamma_1, \mathfrak{I}^{\gamma_1}), (\gamma_2, \mathfrak{I}^{\gamma_2})$ are \emph{companion homotopic (or c-homotopic)} and that $(F,\mathfrak{I}^F)$ is a \emph{ c-homotopy} between $(\gamma_1, \mathfrak{I}^{\gamma_1})$ and $(\gamma_2, \mathfrak{I}^{\gamma_2})$.

Let $\gamma_1, \gamma_2 : [a,b]  \ra \bK \setminus \{0\}$ be two paths with the same endpoints. If there exist a companion
$\mathfrak{I}^{\gamma_1}$ of $\gamma_1$ and a companion $\mathfrak{I}^{\gamma_2}$ of $\gamma_2$ such
that $(\gamma_1, \mathfrak{I}^{\gamma_1}), (\gamma_2, \mathfrak{I}^{\gamma_2})$ are c-homotopic, then we say that
 $\gamma_1$ and  $\gamma_2$ are \emph{weakly c-homotopic}.
\end{definition}

The following simple result will be helpful in the sequel.

\begin{proposition}\label{explanationhomotopy}
Let the continuous map $F: [a, b]\times [c, d] \to \mathbb K\setminus \{0\}$ with companion $\mathfrak{I}^{F}$ be a c-homotopy between $(\gamma_1, \mathfrak{I}^{\gamma_1})$ and $(\gamma_2, \mathfrak{I}^{\gamma_2})$.
Then:
\begin{enumerate}
\item[(i)] the map $\mathfrak{I}^{F}$ is a homotopy between $\mathfrak{I}^{\gamma_1}$ and $\mathfrak{I}^{\gamma_2}$;
\item[(ii)] the homotopy $\mathfrak{I}^{F}$ can be lifted to a homotopy $\cI^{F}$ between a lift $\cI^{\gamma_1}$ of $\mathfrak{I}^{\gamma_1}$ and a lift $\cI^{\gamma_2}$ of $\mathfrak{I}^{\gamma_2}$ in such a way that the canonical form of $F$
\begin{equation}\label{canonical form F}
F(t,s)=x(t,s)+\cI^{F}(t, s)y(t, s)
\end{equation}
is a homotopy between the canonical forms
\[
\gamma_1(t)=x_1(t)+\cI^{\gamma_1}(t)y_1(t)
\]
and
\[
\gamma_2(t)=x_2(t)+\cI^{\gamma_2}(t)y_2(t)
\]
of $\gamma_1$ and $\gamma_2$, respectively;
\item[(iii)] the shadows
\[
x_1(t)+iy_1(t), \qquad x_2+iy_2(t)
\]
of $(\gamma_1, \mathfrak{I}^{\gamma_1})$ and $(\gamma_2, \mathfrak{I}^{\gamma_2})$, respectively, are homotopic in $\mathbb C \setminus\{0\}$.
\end{enumerate}
\end{proposition}
\begin{proof}
The proofs of \emph{(i)} and \emph{(ii)} are a straightforward consequence of Definition \ref{c-homotopy} and of what is stated in Propositions \ref{realization} and \ref{realization2}. Let us prove \emph{(iii)}. To this aim, consider the canonical form of $F$ that appears in (\ref{canonical form F}) and the following continuous maps, for $(t,s)\in [a,b]\times[0,1]$:
\[\begin{array}{l}
\mathcal    {L}_1(t,s)=\Re(F(t,s)),\\
\mathcal{L}_2(t,s)=-\cI^{F}(t, s)\Im(F(t,s))
\end{array}
\]
%
We will prove that $\mathcal{L}=(\mathcal{L}_1, \mathcal{L}_2): [a,b]\times[0,1] \to \R^2\setminus\{(0,0)\}$ is a homotopy between the two given  shadows of $\gamma_1$ and $\gamma_2$.  Indeed, using directly formula \eqref{canonical form F} for the canonical form of $F$,  it is easy to check that on $[a,b]\times[0,1]$,
\[
\begin{array}{l}
\mathcal{L}(t,0)=(\mathcal{L}_1(t,0), \mathcal{L}_2(t,0))=(x_1(t), y_1(t)),\\
\mathcal{L}(t,1)=(\mathcal{L}_1(t,1), \mathcal{L}_2(t,1))=(x_2(t), y_2(t)),\\
\mathcal{L}(a,s)=(\mathcal{L}_1(a,s), \mathcal{L}_2(a,s))=(x_1(a), y_1(a))=(x_2(a), y_2(a)),\\
\mathcal{L}(b,s)=(\mathcal{L}_1(b,s), \mathcal{L}_2(b,s))=(x_1(b), y_1(b)=(x_2(b), y_2(b)).\\
\end{array}
\]
The proof is now complete.

\end{proof}

\begin{example}
To better illustrate the major difference between complex and quaternionic cases, consider the complex curve, defined by
$$
\begin{array}{ll}
  \gamma(t) = 3e^{it}, \, t \in [0,\pi],  &\gamma(t)= -4 + \dfrac{t}{\pi},\, t \in [\pi,3\pi],\\\\
  \gamma(t) = e^{-it}, \, t \in [3\pi,4\pi],& \gamma(t)= - 3+\dfrac{t}{\pi},\, t \in [4\pi,6\pi].
\end{array}
$$
As a complex curve, i.e. with the constant companion $i,$ the curve $\gamma$ has  winding number $0$ and coincides with its own shadow.
As a quaternionic curve, $\gamma$ has a large family of companions $\mathfrak{I}^{\gamma}$;  for example one can consider-
$$
\begin{array}{ll}
  \mathcal{I}^{\gamma}(t) = i, \, t \in [0,\pi], & \mathcal{I}^{\gamma}(t)=  \mathcal{J}(t), \, t \in [\pi,3\pi],\\
  \mathcal{I}^{\gamma}(t) =  -i , \, t \in [3\pi,4\pi], & \mathcal{I}^{\gamma}(t)=  - \mathcal{J}(t),\, t \in [4\pi,6\pi];
\end{array}
$$
and $\mathfrak{I}^{\gamma}(t) = [\mathcal{I}^{\gamma}(t)]$,  where $\mathcal{J}: [\pi,3\pi] \ra \bS$ is an arbitrary continuous
curve with $\mathcal{J}(\pi) = i$ and $\mathcal{J}(3\pi) = -i.$
Correspondingly, the shadow of $(\gamma, \mathfrak{I}^{\gamma}) $ is
$$
\begin{array}{ll}
  \gamma_{\mathfrak{I}^{\gamma}}(t) = 3e^{it}, \, t \in [0,\pi],  & \gamma_{\mathfrak{I}^{\gamma}}(t)= -4 +\dfrac{t}{\pi},\, t \in [\pi,3\pi],\\\\
  \gamma_{\mathfrak{I}^{\gamma}}(t) = e^{it}, \, t \in [3\pi,4\pi], & \gamma_{\mathfrak{I}^{\gamma}}(t)= -3+\dfrac{t}{\pi},\, t \in [4\pi,6\pi]
\end{array}
$$
and so the winding number of $(\gamma, \mathfrak{I}^{\gamma})$ is $1.$ The pairs $(\gamma, i)$ and $(\gamma, \mathfrak{I}^{\gamma})$ are not c-homotopic.
\end{example}

%
%
%
%

The notion of c-homotopy is particularly useful in this 
 setting, because of the following result.
\begin{proposition} \label{c-homotopy separate}
Let $(\gamma_1, \mathfrak{I}^{\gamma_1}), (\gamma_2, \mathfrak{I}^{\gamma_2}) : [a,b]  \ra \bK \setminus \{0\}$ be two paths with companions and with the same endpoints $\gamma_1(a)=\gamma_2(a)=p$ and $\gamma_1(b)=\gamma_2(b)=q$. 
Then the following statements are equivalent :
\begin{enumerate}
\item $(\gamma_1, \mathfrak{I}^{\gamma_1}), (\gamma_2, \mathfrak{I}^{\gamma_2})$ are c-homotopic;
\item  $\mathfrak{I}^{\gamma_1}$ and $\mathfrak{I}^{\gamma_2}$ are homotopic in $\bS/\left\{\pm Id\right\}$, and, in addition, for each of the shadows of $(\gamma_1, \mathfrak{I}^{\gamma_1})$ there is a shadow of $(\gamma_2, \mathfrak{I}^{\gamma_2})$ so that these two  shadows are homotopic   in $\C\setminus \{0\}$.
\end{enumerate}

\end{proposition}
\begin{proof}
Suppose first that \emph{(2)} holds. Then there exist:
\begin{itemize}
\item a homotopy $\bold{G}:[a,b]\times [0,1] \to \bS/\left\{\pm Id\right\}$ between $\mathfrak{I}^{\gamma_1}$ and $\mathfrak{I}^{\gamma_2}$;
\item a lift of  $\bold{G} $, i.e. a homotopy $G:[a,b]\times [0,1] \to \bS$  between a lift $\cI^{\gamma_1}$ of $\mathfrak{I}^{\gamma_1}$ and a lift $\cI^{\gamma_2}$ of $\mathfrak{I}^{\gamma_2}$;
\item  a homotopy $L=(L_1, L_2):[a,b]\times [0,1] \to \R^2\setminus\{(0,0)\}$  between a shadow of $\gamma_1$ and a shadow  of $\gamma_2$ (its ``conjugate" being a homotopy between the corresponding conjugate shadows).
\end{itemize}
In this situation, the map $F:[a,b]\times[0,1] \to \bK \setminus \{0\}$ defined by
\[
F(t,s)=L_1(t, s) +G(t,s)L_2(t,s)
\]
is a  homotopy between $\gamma_1$ and $\gamma_2$. Indeed, $F$ is obviously continuous, and such that, for all $t\in [a,b]$ and all $s\in [0,1]$,
\[ \begin{array}{l}
F(t,0)=L_1(t, 0) +G(t,0)L_2(t,0)=x_1(t)+\cI^{\gamma_1}(t)y_1(t)=\gamma_1(t);\\
F(t,1)=L_1(t, 1) +G(t,1)L_2(t,1)=x_2(t)+\cI^{\gamma_2}(t)y_2(t)=\gamma_2(t);\\
F(a,s)=L_1(a,s) +G(a,s)L_2(a,s)=x_1(a)+\cI^{\gamma_1}(a)y_1(a)=\gamma_1(a)=\gamma_2(a);\\
F(b,s)=L_1(b,s) +G(b,s)L_2(b,s)=x_1(b)+\cI^{\gamma_1}(b)y_1(b)=\gamma_1(b)=\gamma_2(b).
\end{array}
\]
Moreover, the continuous map $G:[a,b]\times[0,1] \to \bS$ defines, by construction, a companion of $F$ given by
\[
\mathfrak{I}^{F}(t,s)=[G(t,s)] = \bold{G}(t,s)
\]
for all $(t,s)\in [a,b]\times[0,1]$. As a consequence, $(F, \mathfrak{I}^{F})$ is a c-homotopy between $(\gamma_1, \mathfrak{I}^{\gamma_1}), (\gamma_2, \mathfrak{I}^{\gamma_2})$.

Let us now suppose that \emph{(1)} holds, i.e. that $(\gamma_1, \mathfrak{I}^{\gamma_1})$ and $(\gamma_2, \mathfrak{I}^{\gamma_2})$ are c-homotopic.
In this case $ \mathfrak{I}^{\gamma_1}$ and $\mathfrak{I}^{\gamma_2}$ are homotopic by definition, and the rest of the assertion follows from Proposition \ref{explanationhomotopy}.

\end{proof}

\begin{proposition}
Let $\gamma : [a,b]  \ra \bK \setminus \{0\}$ be a loop with   companion $\mathfrak{I}^{\gamma}$. Then $(\gamma, \mathfrak{I}^{\gamma})$ is untwisted if, and only if, it is c-homotopic to one of its (closed and conjugate) shadows in $\mathbb C_{\mathfrak{I}^{\gamma}(a)}$.
\end{proposition}

\begin{proof}
%
%

If $(\gamma, \mathfrak{I}^{\gamma})$ is untwisted, then any lift $\cI^{\gamma}$ of the companion $\mathfrak{I}^{\gamma}$ with initial point $\cI^{\gamma}(a)$ is homotopic in $\bS$ to the constant loop $\cI^{\gamma}(a)$, and therefore 
the loop $\gamma$ is c-homotopic to its (closed) shadow in $\C_{ \cI^{\gamma}(a)}$ (see Proposition \ref{untwisted loop}). On the other hand, if the loop with companion $(\gamma, \mathfrak{I}^{\gamma})$ is c-homotopic  to its shadow, then 
the lift of its companion $\mathfrak{I}^{\gamma}$ with initial point $\cI^{\gamma}(a)$ has to be homotopic in $\bS$ to the constant loop $\cI^{\gamma}(a)$. As a consequence the loop $(\gamma, \mathfrak{I}^{\gamma})$ is untwisted by definition.
\end{proof}

%
%

The notion of c-homotopy is suitable to comply with the meaning of the winding number of loops in the setting of $\bK\setminus\{0\}$.
In this panorama, all untwisted tame loops play a special role: any such a loop has an ``intrinsically defined" winding number that depends only on its geometric properties. Indeed, we can state the following result.

\begin{theo}\label{c-omotopia}
Let $\gamma, \delta: [a,b]  \ra \bK \setminus \{0\}$ be two untwisted, tame loops. Then $\gamma$ and $\delta$ are  c-homotopic if, and only if, $\emph{wind}(\gamma)=\emph{wind}(\delta)$.

\end{theo}
\begin{proof}
By Proposition \ref{c-homotopy separate}, $\gamma$ and $\delta$ are  c-homotopic if, and only if, the unique companions $ \mathfrak{I}^{\gamma}$ and $ \mathfrak{I}^{\delta}$ are homotopic and a shadow of $(\gamma, \mathfrak{I}^{\gamma})$  is homotopic to a shadow of $(\delta, \mathfrak{I}^{\delta})$, in $\C\setminus\{0\}$.
According to Definition \ref{windingnew}, the winding number of $(\gamma, \mathfrak{I}^{\gamma})$ (or $(\delta, \mathfrak{I}^{\delta})$)  is defined as the absolute value of the winding number of one of the two (closed) shadows of $(\gamma, \mathfrak{I}^{\gamma})$ (or  $(\delta, \mathfrak{I}^{\delta})$). Therefore the proof is a straightforward consequence of the properties of the fundamental group $\Pi_1(\C\setminus\{0\})\equiv \mathbb Z$, where the class of each loop is determined by its winding number (with respect to zero).
\end{proof}

%
%


The given definition of winding number, which has particularly transparent geometrical meanings, cannot be  adopted as it is in the twisted case, due to the two following results.

\begin{proposition}\label{conjugate endpoints}
Let $\gamma : [a,b]  \ra \bK \setminus \{0\}$ be a loop with a companion $\mathfrak{I}^{\gamma}$. Then $(\gamma, \mathfrak{I}^{\gamma})$ is twisted if, and only if, for any chosen non real initial point of $\gamma$, any shadow associated with $(\gamma, \mathfrak{I}^{\gamma})$ has conjugate endpoints.
\end{proposition}
\begin{proof}
Let $\gamma_{\cI^{\gamma}}: [a,b] \to \C \setminus \{0\}$,
 $
 \gamma_{\cI^{\gamma}}(t)=x(t)+iy(t),
$
be one of the shadows associated with $\mathfrak{I}^{\gamma}$.

If the loop $(\gamma, \mathfrak{I}^{\gamma})$  is twisted, then the lift $\cI^{\gamma}$ is not closed, and hence it has opposite endpoints. Therefore, the path
  \[
  \gamma(t)=x(t)+\cI^{\gamma}(t)y(t)
  \]
being  closed, the continuous function $y: [a,b]\to \R$ is such that $y(a)=-y(b)$. Hence the associated shadow  $
 \gamma_{\cI^{\gamma}}(t)=x(t)+iy(t)
$
has conjugate endpoints.

On the other hand, suppose the associated shadow $\gamma_{\cI^{\gamma}}: [a,b] \to \C \setminus \{0\}$,
 $
 \gamma_{\cI^{\gamma}}(t)=x(t)+iy(t)
$,
has conjugate  nonreal endpoints. Then $y(a)=-y(b) \ne 0$ and, the path
  \[
  \gamma(t)=x(t)+\cI^{\gamma}(t)y(t)
  \]
being  closed by assumption, we obtain $\cI^{\gamma}(a)=-\cI^{\gamma}(b)$ and so the lift $\cI^{\gamma}$ of $\mathfrak{I}^{\gamma}$ is not closed. In conclusion, $(\gamma, \mathfrak{I}^{\gamma})$  is twisted.
\end{proof}


\begin{corollary} \label{realendpoints}
Let $\gamma: [a,b] \ra \bK \setminus \{0\}$ be a loop with companion $\mathfrak{I}^{\gamma}$. Then the two shadows associated with $\mathfrak{I}^{\gamma}$ are closed if, and only if, the endpoints of $\gamma$ are real.
\end{corollary}
\begin{proof}
The proof follows immediately from Proposition \ref{conjugate endpoints}.
\end{proof}

We might be encouraged to think that,  in the case of a loop with companion which is twisted, we should first parameterise the loop in such a way that it has real endpoints (see Proposition \ref{real axis}), and then use Definition \ref{windingnew}. Indeed, this approach gives a weird result, if tested, for instance, in the case of the twisted, tame loop $\lambda$ presented in the next example.

\begin{example}\emph{
Consider the loop $\lambda$  in the hyperplane
        of $\H$  generated by the orthogonal units $\{1,i,j\}$. The path consists of
  several arcs: the arc of  parabola $ t + 1 + t^2(i + j), t \in [-1,1],$
  the segments from $(2,1,0)$ to $(2,1,1)$, from $(2,1,0)$ to $(0,1,0)$
  and  from $(0,0,1)$ to $(0,1,1),$ the halfcircle $\cos t + i
  \sin t, t \in [\pi/2,3\pi/2]$ and the quarter of circle $i\cos t + j
  \sin t, t \in [\pi/2,\pi].$ Let the orientation be such that it
  coincides with the positive orientation of the halfcircle part in
  the plane containing $1,i$. The path intersects the real axis at points $z = 1$ and  $z =
  -1.$
  }
\end{example}

\begin{figure}[H]
$\begin{array}{ccc}
    \includegraphics[width=0.3\textwidth]{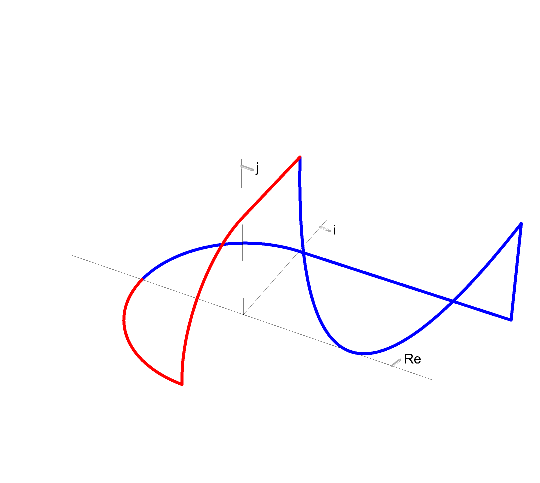}   &
\includegraphics[width=0.30\textwidth]{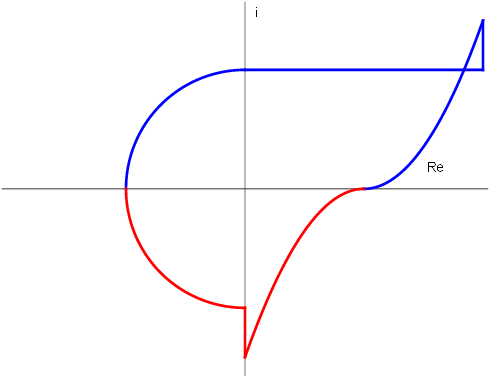} &
\includegraphics[width=0.3\textwidth]{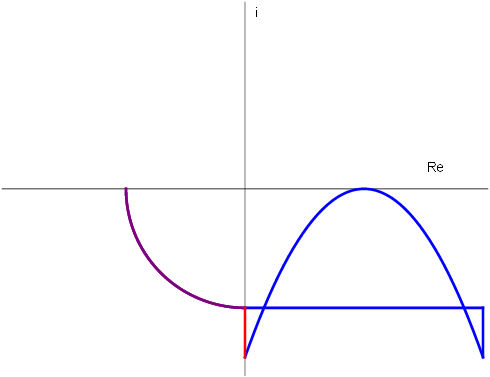}\\
\mbox{\footnotesize(a)}&\mbox{\footnotesize(b)}&\mbox{\footnotesize(c)}\\
& &
\end{array}$
\caption{\footnotesize{From left to right: (a) the path $\gamma$  and two of its shadows (b), (c) \phantom{aaaaaaaaaaaaaaaaa}}}
\label{twistedfigure2}
\end{figure}

In the previous example, the proposed winding number of the twisted,
tame loop $\lambda$ would be $1$ if the loop is parameterised with
real endpoints equal to $1\in \R$ (see Figure \ref{twistedfigure2} (b)). On the other hand, the same
loop
$\lambda$ parameterised with endpoints equal to $-1\in \R$ would have
winding number $0$ (see Figure \ref{twistedfigure2} (c)). What we just
illustrated clarifies that a notion of winding number for twisted,
tame loops in $\bK \setminus \{0\}$ (if it exists) has to be given by
following a different approach.

 In the spirit of the above example and Proposition
  \ref{conjugate endpoints} the definition of the winding number for a
  closed tame twisted loop $\gamma: [a,b] \ra \bK \setminus \{0\}$
  cannot be given by considering
  the change of the argument since
  this  depends on the choice of the initial point.

  Assume that $\gamma$ is a twisted loop in $\bK\setminus\{0\}$ which
  intersects both the positive and the negative real axis;
let $\gamma_{\cI^{\gamma}}: [a,b] \to \C \setminus \{0\}$, $
\gamma_{\cI^{\gamma}}(t)=x(t)+iy(t), $ be one of the shadows
associated with $\gamma$. Let $\arg^{\gamma}(t), t \in [a,b],$ be the
corresponding argument and choose the initial argument so that
$\arg^{\gamma}(a) \in [0, \pi].$
The set $\Delta = \{\arg^{\gamma}(a)
\mbox{ for all possible initial points}\}$ is an interval contained in
$[0, \pi].$ Because the
loop $\gamma$
is twisted, the argument at $b$ is
$\arg^{\gamma}(b) = 2 n \pi - \arg^{\gamma}(a)$ and hence
$$\arg^{\gamma}(b)-\arg^{\gamma}(a)=2 n \pi - 2\arg^{\gamma}(a)$$ and
is not an integer multiple of $2\pi$ unless $\arg^{\gamma}(a) =
0,\pi.$ Even if we set the initial point to be real, so that the
change of argument is $2n\pi,$ the number $n$ can have more than one
value as shown in the following example.

\begin{example}{\em
  Let $\gamma_1$ be the positively oriented unit circle with initial
  point $-1$ and companion $i$ and define $\gamma_2(t):=\cos (t)+i\sin
  (t)+j(\cos (t)+1), t \in [-\pi, \pi]$.  Choose the lift of the
  companion $\mathcal I_2$ of $\gamma_2$ so that $\mathcal I_2(-\pi) =
  i$ and $\mathcal I_2(-\pi) = -i.$ Let
  $\gamma = \gamma_1^m \cdot \gamma_2$  denote
  the loop composed first of $m$ copies of
  $\gamma_1$ followed by a copy of $\gamma_2.$ If
  the initial point is assumed to be
  the point $-1$ on the first copy of $\gamma_1$, then the change of
  the argument is $2\pi m.$ If the initial point is the point $-1$ on
  the second copy of $\gamma_1$,
  then the $m-1$ copies of $\gamma_1$
  before $\gamma_2$ give the
  winding number $m-1$, but then the curve $\gamma_2$ reverses the
  orientation so the last copy of $\gamma_1$ has negative orientation
  with respect to the unit $-i,$ hence the winding number is $m-2.$
  Starting at $-1$ on the third copy
 of $\gamma_1$,
  would therefore result in the
  winding number $m-4$ and so forth.
  }
\end{example}

A few words seem now appropriate, to present a suggestive geometrical explanation of the reason why the notion of winding number as given in the case of untwisted loops does not work for the case of twisted loops. Indeed, consider an untwisted loop $\gamma: [a,b] \ra \bK \setminus \{0\}$
\[
  \gamma(t)=x(t)+\cI^{\gamma}(t)y(t)
  \]
If we regard all points $\{x(t)\} _{t\in [a,b]}$ as distinct points except the endpoints, such a $\gamma$ has values in the surface $S_\gamma=\{x(t)+\cI^{\gamma}(t)s : t \in [a,b], s\in \mathbb R\}$; since $\gamma$ is untwisted, then $\cI^{\gamma}: [a,b]\to  \bS$ is a loop, and hence it is homotopic to the constant loop $\cI^{\gamma}(0)=\cI^{\gamma}(1)$. As a consequence,  the surface $S_\gamma$ is homeomorphic to a  twodimensional  cylinder. Therefore there is a notion of $\gamma(t)$ being a point of this surface lying on one side or the other of the ``real axis''  formed by the points $\{x(t)\} _{t\in [a,b]}$, and hence a notion of winding number with respect to the origin becomes possible:  the situation reduces, naively speaking, to a  planar one. If instead $\gamma: [a,b] \ra \bK \setminus \{0\}$ is twisted, then the path $\cI^{\gamma}: [a,b]\to  \bS$ has antipodal endpoints, and the surface $S_\gamma$ turns out to be homeomorphic to a Moebius strip. In this last situation, the lack of orientability seems to exclude the possibility of defining coherently a winding number for the loop $\gamma$.



\section{Obstructions to the existence of lifts of a path}\label{sec4}
In this section we present  sufficient conditions for a path  to have a lift, a companion and to be tame.

As already mentioned, if the path $\gamma: [a,b] \ra \bK
\setminus\{0\}$ misses the real axis, then the lift to $\mathscr{E}_{\bK}^+$always exists. On
the other hand, if $\gamma([a,b])\subset \mathbb{R}$, then,
necessarily either $\gamma([a,b])\subset \mathbb{R}^- $ and we have
the lifts of the form $\Gamma(t) = \log|\gamma(t)| + I (2k+1)\pi$, or
$\gamma([a,b])\subset \mathbb{R}^+ $ and then we have the lifts of the
form $\Gamma(t) = \log|\gamma(t)| + I 2k\pi$ for any $I \in \bS.$ From
now on assume that $\gamma([a,b])$ is not entirely contained in the
real axis but it intersects it.

\begin{definition} For a  path $\gamma: [a,b] \ra \bK \setminus\{0\}$  we define the set $T :=\gamma^{-1}(\R)$ to be the
  {\em obstruction set} (for the lift of $\gamma$)  and its points as {\em obstruction parameters}.
\end{definition}

It is clear that the necessary assumption for a lift of $\gamma$ to
$\mathcal E_{\bK}^+$ to exist is the requirement that $\gamma$ has a
lift on a neighbourhood of every parameter $t,$ in particular, for
each $t \in T.$ It turns out that the existence of local lifts does
not necessarily imply the existence of a global lift; recall that
complex curves avoiding $0$ always have local and global lifts.\\

In what follows, we start establishing the conditions on the
behaviour of $\gamma$ locally near its obstruction parameters in order to guarantee the existence first of local lifts and local companions and then of a global lift and a global companion.

As these conditions depend on the structure of the obstruction set, we start  by considering paths with a finite obstruction set.

\begin{definition}\label{TameDefinition}
      Let the path $\gamma(t)=x(t)+Y(t): [a,b] \ra \H \setminus \{0\}$ be such that $T = \gamma^{-1}(\R) = \{a \leq t_1 <\ldots < t_p \leq b\}. $
      Consider the limits
     \begin{equation}\label{limits}
       \lim_{t\ra t_s^{\pm}}\frac{Y(t)}{|Y(t)|}.
     \end{equation}
     Let $t_s \in (a,b).$ Then
     \begin{trivlist}{}{}
        \item[1)] $\gamma$ is {\em tame} at $t_s$ if both limits
        are either equal or opposite. In particular,
        if these limits are opposite, then the parameter  $t_s$ is called a {\em flip}, whereas if they are the same it is called {\em a bounce};
        \item[2)] $\gamma$ is {\em semi-tame} at $t_s$ if it is not tame at $t_s$ but both limits in (\ref{limits}) exist;
        \item[3)] $\gamma$ is {\em not tame} at $t_s$ if at least one
          of the limits in (\ref{limits}) does not exist.
      \end{trivlist}
       If $t_s = a$ (resp. $t_s = b$) then the path is tame at $t_s$ from the right (left)
      if the right (left) limit in (\ref{limits}) exists and not tame in all other cases.

      If, in addition, the path $\gamma$ is closed, we adapt the definition of tameness at the endpoints in the natural way. In particular,
      $\gamma$ is semi-tame at $a \simeq b$ if it is tame at $a$  from the right and at $b$ from the left. If the limits are the same, then $a \simeq b$ is called a {\em bounce} and if they are opposite it is called a {\em flip}. In all other cases $\gamma$ is not tame at $a \simeq b.$
\end{definition}

\begin{remark}{\em
 A path $\gamma$ in the Example \ref{esempio1} (b) does not have the limit  (\ref{limits}) at $t =0.$}
\end{remark}


\begin{remark}
{\em The definition of tameness of $\gamma$ at parameters $\{t_1,\ldots, t_p\}$ means precisely that the projectivized imaginary unit function $[{\mathcal J}(\gamma(t))]$ defined on
$[a,b] \setminus \{t_1,\ldots, t_p\}$ has a continuous extension to $\{t_1,\ldots, t_p\}$.}
\end{remark}

\noindent The proposition below gives a motivation for the previous definitions.

\begin{proposition}\label{lifts} Let $\gamma(t): [a,b] \ra \bK \setminus \{0\}$ be a path with finite obstruction set
$\gamma^{-1}(\R) = \{a \leq t_1 <\ldots < t_p \leq b\}. $ Then $\gamma$ is tame if and only if it is tame at each $t_s, s = 1,\ldots,p.$

If $\gamma$ is a loop, then it is a tame loop if and only if it is tame at each $t_s, s = 1,\ldots,p.$
\end{proposition}


\begin{proof}
By assumption, the projectivized imaginary unit function $[{\mathcal J}(\gamma(t))]$ defined on
$[a,b] \setminus \{t_1,\ldots, t_p\}$ has a continuous extension to $[a,b].$
\end{proof}

If $\gamma$ is a tame loop, the lift to $\mathcal E_{\bK}^+$ exists by Proposition \ref{TameLift1}.
However, we want to present also a constructive proof, because we will
use the same techniques to obtain lifts of non tame paths and to explain the definition of winding number through local data on the obstruction set.

Without loss of generality we assume that $a \ne t_1, b \ne t_p.$  Consider the intervals $I_0=[a=:t_0,t_1], I_1=
[t_1,t_2],\ldots,I_p=[t_p, t_{p+1}=:b]$  and denote the
restrictions of $\gamma$ on $I_s$ by $\gamma^s := \gamma|_{I_s}.$ The existence
of limits (\ref{limits}) provides, for any $s = 0,\ldots,p,$
continuous extensions of all functions
$\arg_k(\gamma^s)(t),\cI_k(\gamma^s)(t)$ to the endpoints of $I_s.$

Choose an arbitrary $k_0 \in \Z.$ Setting $\arg_{k_0} (\gamma^0(t))=:\arg^{\gamma}(t)$ and $\cI_{k_0}(\gamma^0(t))=:\cI^{\gamma}(t)$ we define continuous functions $\arg^{\gamma}$ and  $\cI^{\gamma}  $ on
       $[a,t_1].$  We set $\Arg^{\gamma}(t):=\arg^{\gamma}(t)\cI^{\gamma}(t)$ and define the lift
        $\Gamma^0 :=(\gamma^0, \Arg^{\gamma^0}).$
   Consider the endpoint $\gamma(t_1).$ If it is a flip, then we
   define 
   \begin{equation*}
 k_1 :=
\left\{
    \begin{array}{ll}
        k_0 + 1, & \text{if }  \arg_{k_0}(\gamma^0(t_1)) = (k_0+1)\pi,\\
        k_0 -1, & \text{if }  \arg_{k_0}(\gamma^0(t_1)) = k_0\pi.
    \end{array}
\right.
\end{equation*}
%
   If it is a bounce then we set
   $k_1:= k_0.$ By setting $\arg^{\gamma}(t):= \arg_{k_1}
   (\gamma^1(t))$ and $\cI^{\gamma}(t):=\cI_{k_1}(\gamma^1(t))$ we
   extend the functions $\arg^{\gamma}$ and $\cI^{\gamma} $
   continuously to $[a,t_2].$ We extend the above functions to $[a,b]$ by repeating this process.

\begin{proposition}\label{flips}
  A tame loop $\gamma$ with $\gamma^{-1}(\R) = \{a \leq t_1 <\ldots < t_p \leq b\}$ has  an even number of flips if and only if $\gamma$ is untwisted.
\end{proposition}
\begin{proof}
Assume that the loop does not have flips. Then $\cI^{\gamma}$ equals
$\cI_{k_0}\circ\gamma$ for some $k \in \Z$
and so it is obviously a loop.

For the case of a loop with flips, we assume, without loss of generality, that $k_0 = 0,$
so we have started with the principal branch and moreover, we also assume that the
parameterization $\gamma: [a,b] \ra \H$ is such that $\gamma(a) \in
\R.$

As in the previous proof, all the functions
$\arg_0(\gamma^s)(t),\cI_0(\gamma^s)(t)$ have continuous extensions to
the endpoints of $I_s.$

If $a$ is a bounce, then the imaginary unit at $\gamma(b)$ is
the same as the one at $\gamma(a)$, i.e. $\cI_0(\gamma^0)(a)= \cI_0(\gamma^p)(b).$ The even
number of flips ensures that the sign of the imaginary unit at the
endpoint remains the same with respect to the one at the principal
branch.

If the initial point is a flip, then the imaginary unit function at
endpoint has the opposite sign with respect to the one at the initial
point, $\cI_0(\gamma^0)(a)= -\cI_0(\gamma^p)(b),$ and to end up with
the same sign there must be an odd number of additional flips
following the first one to ensure that the sign of the unit at the
endpoint remains the same with respect to the one at the principal
branch.
\end{proof}

The proofs of Propositions \ref{lifts} and \ref{flips} show that once
the lift near the initial point is chosen, only the flips are relevant
for the determination of the lift near the endpoint; bounces can be
discarded. This enables us to calculate the change of argument and the
winding number out of local data at the intersections of $\gamma$ with
the real axis. To determine the change of the argument we introduce a
notion of {\em signature}.

\begin{definition}
Let $\gamma:[a,b]  \ra \bK \setminus \{0\}$ be a given  path with $\gamma^{-1}(\R) = \{a \leq t_1 < \ldots < t_p \leq b\}$,  with  points of $\gamma^{-1}(\R) \cap (a,b)$ all  tame.  Let $a < \xi_1 <\ldots <\xi_m < b$  be those parameters
in $\gamma^{-1}(\R)$ which are flips. The signature $\sigma(\gamma)$ is defined by
$$
  \sigma(\gamma):= \sum_{l = 1}^m\sign(\gamma(\xi_l))(-1)^l.
$$
If there are no flips, then we define $\sigma(\gamma):=0.$
\end{definition}
\noindent
The connection between the signature and the change of argument is described  in the following

\begin{proposition}\label{signature}
Let $\gamma:[a,b] \ra \bK \setminus \{0\}$ be a tame  path  with
$\gamma^{-1}(\R)=\{a_1\leq t_1 < \ldots < t_p \leq b\}$ with all the parameters  $\gamma^{-1}(\R) \cap (a,b)$ tame. Assume that a lift $\Gamma$
of $\gamma|_{[a,t_1]}$ in $\mathcal{E}_\bK^+$
exists
and equals $\Gamma(t) = (\gamma(t), \Arg_{k_0}^{\gamma}(t)) \in \mathcal{E}^+_\bK, t \in [a,t_1]$
 for some $k_0 \in \Z$.
Then the lift of $\gamma|_{[t_p,b)}$ is given by $(\gamma(t),\Arg^{\gamma}_{k_0 + (-1)^{k_0}\sigma(\gamma)}(t)), t \in [t_p,b).$
\end{proposition}
\begin{remark}{\em  If $\gamma(b) \in \R$, then a lift of $\gamma$ on $[a,b)$  can be extended continuously to $b$ if and only if $b$ is tame from the left.}
\end{remark}
\begin{corollary} Let $\gamma$ be as in  Proposition \ref{signature} and let $\gamma(a) = \gamma(b) \not\in \R.$ Then $\sigma(\gamma)$ is even if and only if $\gamma$ is tame and untwisted.
 If this is the case, then
\[\omega(\gamma) = |\sigma(\gamma)|/2.\]
\end{corollary}

\noindent{\em Proof of Proposition \ref{signature}. } Assume that $k_0 = 0,$ so $\arg(\gamma(a)) \in (0,\pi)$ and let the sequence of signs of flips be alternating starting with $-1,$ i.e. $-1,1,-1,\ldots .$
Then $\arg^{\gamma}(\gamma(t))$ increases when the path $\gamma$ crosses the real axis,
so $k$ increases by $1$ at each flip (because the sign of the flip changes); altogether, this occurs
$\sum_{l=1}^m (-1)^l(-1)^l = \sum_{l = 1}^m\sign(\gamma(\xi_l))(-1)^l = \sigma(\gamma)$ times.  This coincides with winding around
the origin of the shadow in the positive
direction. If the sequence of signs starts with $1,$ then
$\arg^{\gamma}(\gamma(t))$ decreases when the path $\gamma$ crosses the positive real axis and
this results in the translation of the interval $[0,\pi]$ by $\pi \sum_{l=1}^m (-1)^{l-1}(-1)^l=\pi \sigma(\gamma).$

To prove the general assertion it suffices to show what happens if the sequence is not alternating at one position.

Assume that we insert in the alternating sequence $-1,1-1\ldots$ the
integer $1$ in the second position, so the sequence is no longer
alternating: $-1,1,1,-1,\ldots.$ This means that we have started from
the upper half-plane, crossed the negative real axis, then the
positive real axis with the arguments in $[2\pi, 3\pi].$ Then we have
crossed the positive real axis again, hence the choice of argument at
this intersection must be $2\pi.$ Because the point is a flip, this
means that the argument decreases and keeps decreasing till the end.
This is faithfully reflected in the sequence $s_l =(-1)^l
\sign(\gamma(\xi_l)),$ because it equals $1,1,-1,-1,...$ and so the sum $\sum_{l = 1}^m\sign(\gamma(\xi_l))(-1)^l = \sigma(\gamma)$
multiplied by $\pi$ corresponds with the total translation of the
initial interval for the $\Arg$.

Similarly, if we insert $-1$ on the second position, this means that
we have crossed the negative real axis and we have the argument in
$[\pi, 2\pi]$ but then we have returned to the negative real axis and
in order to have the argument continuous, at the second crossing the
argument $\pi$ must be chosen and because we have a flip, the argument
$\arg^{\gamma}(\gamma(t))$ decreases and keeps decreasing till the
end.  The corresponding sequence $s_1,s_2,\ldots$ now equals
$1,-1,-1,-1,\ldots$ and $\sum_{l = 1}^m s_l = \sigma(\gamma).$

The proof for $k_0$ even is the same.  If $k_0$ is odd, this coincides with considering the conjugate shadow and hence  reversed
orientation compared, to $k_0$ even,  so the signature has to be multiplied by $-\pi$ to get the total translation of the initial interval for the $\Arg$.
\hfill$\Box$\\

In practice this means that once the sequence of $\pm 1$-s is given,
we start by cancelling the pairs of the same numbers until we end
up with an alternating sequence. The number of elements multiplied by
minus the first element is the signature.

If the path $\gamma : [a,b] \ra \bK \setminus \{0\}$ is closed,
i.e. $\gamma(a)=\gamma(b)$, then we identify points $a$ and $b$ of $[a,b]$ and
consider the parameterization as $\gamma : S^1 \ra \bK \setminus
\{0\},$ so there is no distinguished initial point. Therefore, in this
case we require that for each $s \in S^1$ there exists a neighbourhood
of $U_s$ of $s$ in $S^1$ such that the lift of $\gamma$ exists on
$U_s.$

\begin{definition}\label{LiftsClosedCurves}
  Let $\gamma : S^1 \ra \bK \setminus \{0\}$ be a continuous loop. Then a continuous function
  $\Gamma: i\R \ra \mathcal{E}^+$ is {\em a lift of $\gamma$} if
  the following diagram commutes:
$$\xymatrix
   { i\R \ar[d]^{\exp}  \ar[rr]^{\Gamma} & & \Gamma(i\R) \ar[d]^{\pro\mathrm{}} \subset \mathcal{E}_\bK^+   \\
     S^1   \ar[rr]^{\gamma}  & & \gamma(S^1) \subset \bK \setminus\{0\}
   }
$$
\end{definition}

\begin{remark} {\em A loop with companion always has a (not necessarily closed) lift in the sense of Definition \ref{LiftsClosedCurves}.  The loop presented in  Figure \ref{meridians} does not have a lift in the sense of Definition \ref{LiftsClosedCurves}.
The curve is defined by $\gamma(t) = \cos(t) + i (\sin(t) - t/10)$ for
$t \in [0,\pi]$ and $\gamma(t) = \cos(2\pi-t) -i(2\pi-t)/10 + j
\sin(2\pi-t) $ for $t \in [\pi, 2\pi].$ However, starting from any
point with $\gamma(t) \in \R \times i\R^+$ and using the principal
branch one can obtain the local lift of $\gamma$ and prolong it to the
interval $[0,2\pi].$ }
\end{remark}

\begin{figure}[h!]
\centering
 \includegraphics[scale=0.6]{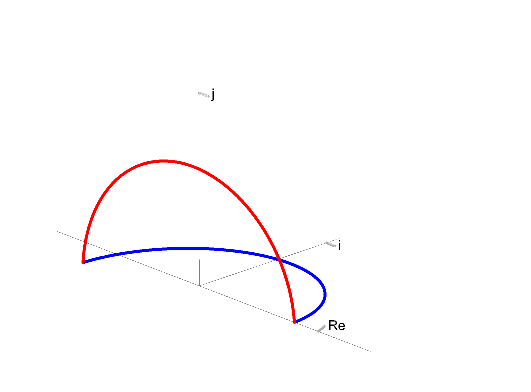}
 \begin{center}\caption{\footnotesize A loop without a lift in the sense of Definition \ref{LiftsClosedCurves}.  }
 \label{meridians}
 \end{center}
\end{figure}

\begin{corollary}\label{ClosedNonTame}
Let $\gamma:[a,b] \ra \bK \setminus \{0\}$ be a continuous loop with
$\gamma^{-1}(\R)$ nonempty and  assume it is  not tame at least  at one of the obstruction parameters.
Let $a = \xi_1 < \ldots <\xi_m = b \simeq a$ are all the
obstruction parameters where $\gamma$ is not tame and assume, moreover, that  $\gamma(\xi_k)>0.$  Then a lift of $\gamma$ in $\mathcal{E}^+_{\bK}$ exists if and only if
$\sigma(\gamma|_{[\xi_l,\xi_{l+1}]}) \in \{0,-1\}$ for each $l =
1,\ldots, m-1.$
If it exists, the lift is a loop.
\end{corollary}

\begin{proof} Because $\gamma$ is not tame at $\xi_l$ we can only choose either $k_0 = 0$  or $k_0 = -1$
  and lift the curve in a neighborhood of the point $\gamma(\xi_l)$ to $\mathcal{E}^+_{\bK}$ using the principal branch of the logarithm.
  Assume that we have  chosen $k_0=0$. Then on $[\xi_{l+1} -\delta,
    \xi_{l+1})$ for some small $\delta>0$ we may only have $k = 0,-1,$
    hence the signature can be either $0$ or $-1$ in order to be able
    to extend the lift to $\xi_{l+1}.$ If we have $k_0 = -1,$ then,
    since we have to end up with $k= 0,-1$ near $\xi_{i+1},$ the
    condition is $(-1)^{k_0}\sigma(\gamma|_{[\xi_l,\xi_{l+1}]}) \in
    \{0,1\}$ hence $\sigma(\gamma|_{[\xi_l,\xi_{l+1}]}) \in \{0,-1\}.$
\end{proof}

When we do not have additional information about the
set $\gamma^{-1} (\mathbb{R})$,
we must assume that the continuous lift of $\gamma$
exists on a neighbourhood $U$ of $\gamma^{-1}(\R) \cap
(-\infty,0)$. This means that the path $\gamma|_U$  has a companion, since
the restriction of the function $\arg^{\gamma}$  to $U$ is not vanishing.
Recall that, on a neighbourhood of
$\gamma^{-1}(\R) \cap (0,\infty)$, the principal branch of the logarithm
is well-defined
and hence a lift of $\gamma$ always exists. This does not imply that a global lift exists.

We now proceed with the detailed description of the possible situations
when $\gamma$  has a companion  on a neighborhood of real points and omit the (trivial) case $\gamma([a,b]) \subset \R$.
\begin{proposition} Let $\gamma:[a,b]\rightarrow \bK \setminus \{0\}$ be a  path.  Then $\gamma$ has a companion if and only if it has a companion on a neighbourhood of the obstruction set. The same holds for a loop $\gamma$ with $\gamma(a) \not \in \R$.
\end{proposition}

In the sequel we explain how to extend the notion of signature to paths with infinite obstruction set. Since $\gamma([a,b])$ is compact, there are only finitely many connected components
of $\gamma([a,b])\setminus \R$ with endpoints of opposite
sign.

\begin{definition} Let $L_1,\ldots, L_m$ be all the connected components of $\gamma
([a,b]) \setminus \R,$  $L_l(t) = \gamma(t), t \in (s_l,e_l) \subset
  [a,b]$ satisfying $\gamma(s_l)\gamma(e_l) < 0$ and $a\leq
  s_l<e_l\leq s_{l+1}< e_m\leq b$, $l=1,\ldots, m$.  We call the
  components {\em the big arcs} and the subdivision $a\leq s_l<e_l\leq
  s_{l+1}< e_m\leq b$, $l=1,\ldots, m$ {\em the induced
    subdivision}. The intervals $[e_l, s_{l+1}]$ are called {\em
    obstruction intervals}. If $\gamma$ is closed, then we identify
  $a$ and $b,$ $e_0:=e_m, s_{m+1}:=s_1$ and define also $[e_0,s_1]$ as
  the obstruction interval.
\end{definition}
Because  $\gamma([e_l,s_{l+1}])$ misses either the positive or the
negative real axis, we define the sign of the obstruction interval as follows.
\begin{definition}
  If $\gamma([e_l,s_{l+1}]) \cap (-\infty, 0) = \varnothing,$
  then $\sign([e_l,s_{l+1}]) = 1$;
otherwise, if $\gamma([e_l,s_{l+1}]) \cap (0,\infty) = \varnothing,$ then
$\sign([e_l,s_{l+1}]) = -1.$
\end{definition}

Extend the domains of definition of each $L_l$ to its
 endpoints and let
$$\mathcal
I^l(t):= \mathcal J(\gamma(t)), t \in (s_l,e_l) \mbox{ and }
I_l^s:= \lim\limits_{t
  \ra s_l^-}\mathcal J^l(t), I_l^e:= \lim\limits_{t \ra e_l^+}\mathcal
J^l(t)$$
be the imaginary units of $L_l$ at its endpoints, if the
limits exist.

\begin{definition}
Let $\gamma:[a,b]\rightarrow \bK \setminus \{0\}$ be a  path  with companion $\mathfrak I$ with lifts $\pm \mathcal I.$ Let
 $a \leq s_1 < e_1 \leq s_2 < e_2 <\ldots \leq s_m < e_m \leq b$ be the induced subdivision and $L_l$ the big arcs with limits
 $I_l^e$ and $I_l^s,l = 1,\ldots, m$.

The interval $[e_l,s_{l+1}], 1 \leq l \leq m-1,$ is {\em a bounce with respect to $\mathfrak I$ } if $\mathcal I$ (or $-\mathcal I$)
 satisfies $\mathcal I(e_l) = \pm I_j^e, \mathcal I(s_{l+1})= \pm I_{l+1}^s $ and a
 {\em a flip with respect to $\mathfrak I$} if $\mathcal I$ (or $-\mathcal I$)  satisfies $\mathcal I(e_l) = \pm I_l^e, \mathcal I(s_{l+1})= \mp I_{l+1}^s.$
\end{definition}

\begin{remark}{\em
    If $\gamma$ has a companion and $\gamma([e_l,s_{l+1}]) \cap \R$ contains an open set then $\gamma$ always  has a companion that makes it  a bounce and a companion that makes it a flip. If the interval $[e_l, s_{l+1}]$
    reduces to a point, then the definition of tameness for intervals
    coincides with the definition of
    tameness for points.}
\end{remark}
We can now extend the definition of signature also to this general case.
\begin{definition}\label{DefSignature}
  Let $\gamma:[a,b]\rightarrow \bK \setminus \{0\}$ be a  path with the induced  subdivision $a \leq s_1 < e_1 \leq s_2 < e_2 <\ldots \leq s_m < e_m \leq b$ and a companion $\mathfrak I.$
 Let $1 \leq j_1 < \ldots < j_k \leq m$  be the indices for which the intervals $[e_{j_i}, s_{j_i + 1}]$
 are flips. The signature $\sigma(\gamma,{\mathfrak I})$ with respect to the companion $\mathfrak I$ is defined as
$$
  \sigma(\gamma,{\mathfrak I}):= \sum_{l = 1, j_k \ne m}^k\sign([e_{j_l}, s_{j_l+1}])(-1)^l.
$$
  If $\gamma$ is a  loop, then  we define the 
 circular
  signature with respect to ${\mathfrak I}$ to be
$$
  \sigma^c(\gamma, {\mathfrak I}):= \sum_{l = 1}^k\sign([e_l,s_{l+1}]))(-1)^l.
$$
If there are no flips, then we define $\sigma(\gamma,{\mathfrak I}):=0, \sigma^c(\gamma,{\mathfrak I}):=0.$
\end{definition}

The following are straightforward generalizations of Proposition \ref{signature} and Corollary \ref{ClosedNonTame}
\begin{proposition}\label{signature2}
 Let $\gamma:[a,b]\rightarrow \bK \setminus \{0\}$ is a  path with the companion ${\mathfrak I}$ and the induced  subdivision $a=e_0 \leq s_1 < e_1 \leq s_2 < e_2 <\ldots \leq s_m < e_m \leq b=s_{m+1}.$
Assume that  a lift $\Gamma$ of $\gamma$ in $\mathcal{E}_\bK^+$
is given by $\log_{k_0},$   $k_0 \in \Z$ on  $[s_1-\delta, s_1]$ for some $\delta > 0.$
The lift on $[s_m,e_m]$ is given by $k:= k_0 + (-1)^{k_0}\sigma(\gamma).$
\end{proposition}

To define the winding number for a closed  curve we have to take into account also the last interval $I_m$ and hence consider the closed signature.

\begin{corollary} Let $\gamma$ be a loop and $\sigma^c(\gamma)$ even. 
Then $\omega(\gamma,{\mathfrak I}) = |\sigma^c(\gamma,{\mathfrak I})|/2.$
\end{corollary}

\begin{corollary}\label{ClosedNonTame2}
Let $\gamma:[a,b]\rightarrow \bK \setminus \{0\}$ be a loop with the induced  subdivision $a \leq s_1 < e_1 \leq s_2 < e_2 <\ldots < s_m < e_m \leq b.$

 Let $1 \leq j_1 < \ldots<j_k \leq m$  be the indices for which $\gamma,$ restricted to the neighbourhoods of the intervals $J_{j_l}:=[e_{j_l}, s_{j_{l + 1}}] \subset (0,\infty),$
 does not have a companion and assume $\gamma$ has a companion on a neighbourhood of the closure of $[a,b] \setminus \cup_{l = 1}^k J_{j_l}.$
 Then a lift of $\gamma$ in $\mathcal{E}_\bK^+$ exists if and only if
$\sigma(\gamma|_{[s_{j_l+1},e_{j_{l+1}}]}) \in \{0,-1\}$ for each $l =
1,\ldots, m-1.$
If it exists, the lift is a loop.
\end{corollary}

\end{document}